\documentclass[]{article}
\usepackage[english]{babel}
\usepackage{amssymb}
\usepackage{amsmath,amsthm, dsfont}
\usepackage{subcaption}
\usepackage{bm}
\usepackage{tabularx,array}
\usepackage{graphicx, url}
\usepackage{enumitem} 

\usepackage{mathtools}
\usepackage{diagbox,hhline}
\usepackage{tabularray}
\usepackage{authblk}
\usepackage{comment}

\usepackage{color}

\newcommand{\arnaud}[1]{\todo[color=green!20]{{ #1}}}

\newcommand{\modar}{\color{black}}
\newcommand{\revarn}{\color{black}{}}
\newcommand{\revK}{\color{black}{}}

\usepackage[backend=bibtex,url=false,doi=false,isbn=false]{biblatex}
\renewbibmacro*{urldate}{}

\theoremstyle{plain}
\newtheorem{theorem}{Theorem}[section]
\newtheorem{lemma}[theorem]{Lemma}

\newtheorem{proposition}[theorem]{Proposition}

\newtheorem{remark}[theorem]{Remark}

\theoremstyle{definition}
\newtheorem{assumption}{}

\DeclareMathOperator*{\argmin}{arg\,min}
\newcommand{\one}{\mathds{1}}

\newcommand{\E}{\mathbb{E}}

\newcommand{\id}[1]{\ensuremath{\mathds{1}_{#1}}}
\renewcommand{\P}{\mathbb{P}}

\newcommand{\ve}{\varepsilon}

\usepackage{suffix}
\DeclarePairedDelimiter{\abs}{\lvert}{\rvert}
\DeclarePairedDelimiter{\norm}{\lVert}{\rVert}
\DeclarePairedDelimiter{\pairangle}{\langle}{\rangle}
\newcommand{\bracket}[3][]{\pairangle[#1]{{#2,#3}}}
\WithSuffix\newcommand\bracket*[2]{\pairangle*{{#1,#2}}}

\title{Drift estimation for rough processes under small noise asymptotic : QMLE approach}

\author[1]{Arnaud Gloter}
\author[2,3]{Nakahiro Yoshida}
\affil[1]{Laboratoire de Math\'ematiques et Mod\'elisation d'Evry, Universit\'e d'Evry
	\footnote{
		Laboratoire de Math\'ematiques et Mod\'elisation d'Evry, CNRS, Univ Evry, 
		Universit\'e Paris-Saclay, 91037, Evry, France. e-mail: arnaud.gloter@univ-evry.fr}
}
\affil[2]{Graduate School of Mathematical Sciences, University of Tokyo
	\footnote{Graduate School of Mathematical Sciences, University of Tokyo: 3-8-1 Komaba, Meguro-ku, Tokyo 153-8914, Japan. e-mail: nakahiro@ms.u-tokyo.ac.jp}
}
\affil[3]{Japan Science and Technology Agency CREST
}

\begin{document}

\maketitle

\begin{abstract}
	We consider a process $X^\ve$ solution of a stochastic Volterra equation with an unknown parameter $\theta^\star$ in the drift function. The Volterra kernel is singular {\revK near zero, exhibiting a behavior comparable to  
		$K_0(u)=cu^{\alpha-1} \id{u>0}$ with $\alpha \in (1/2,1)$.}
	It is assumed that the diffusion coefficient is proportional to $\ve \to 0$.	
Based {\revarn on discrete observations, with a mesh size $h\to0$, of the Volterra process,}
we construct a  Quasi Maximum Likelihood Estimator.
The main step is to assess the error arising in the reconstruction of the path of a semimartingale  from the inversion of the Volterra kernel. We show that this error decreases as $h^{1/2}$ 
{\revarn regardless of the value of $\alpha$.} Then, we can introduce an explicit contrast function, which yields an efficient estimator when $\ve \to 0$. 
\end{abstract}

\section{Introduction}

Recently, many works have been devoted to differential stochastic equations with fractional derivative. Such equations have the structure
\begin{equation}\label{Eq : Caputo form}
	{}^CD_{0+}^\alpha X_t=b(X_t)+a(X_t)\frac{dB_t}{dt},
\end{equation}
where ${}^CD_{0+}^\alpha$ is the Caputo derivative of order $\alpha \in (1/2,1)$ and $B$ is a standard Brownian motion. These equations can be written in integral form by inverting the Caputo derivative. This manipulation gives rise to a special case of Volterra equations like the one given in \eqref{Eq : Volterra SDE_intro} with a specific kernel $K$ (e.g. see \cite{huongWellposednessRegularitySolutions2023a}). 
The fractional time derivative equation \eqref{Eq : Caputo form} is a suitable model to describe memory properties arising in several  fields {\revarn such as physics, mathematical finance, or life- science.} Among references dealing with applications of Volterra equation is \cite{bakerPerspectiveNumericalTreatment2000a} and references therein for the deterministic case or
{\revarn  \cite{abijaberAffineVolterraProcesses2019,barndorff-nielsenModellingEnergySpot2013,euchCharacteristicFunctionRough2016}}
in the context of energy or financial modeling using stochastic models.
%

The probabilistic properties of the stochastic Volterra models driven by Brownian motions are first studied in the seminal papers
\cite{bergerVolterraEquationsIto1980, bergerVolterraEquationsIto1980b},  for non singular kernels. The model is extended in many ways, by considering more general semimartingale \cite{protterVolterraEquationsDriven1985},  singular kernels
 \cite{cochranStochasticVolterraEquations1995}, non Lipschitz  coefficients \cite{wangExistenceUniquenessSolutions2008, promelStochasticVolterraEquations2023},
infinite dimensional settings
{\revarn \cite{wangAsymptoticBehaviorStochastic2016a, zhangStochasticVolterraEquations2010},} or connection with rough path theory \cite{promelParacontrolledDistributionApproach2021}.

The case of Volterra equations with singular kernels, corresponding to \eqref{Eq : Caputo form}, generates so-called ``rough'' processes, as the solution $X$ admits a H\"older smoothness index $\alpha-1/2$  smaller than the Brownian regularity index $1/2$.
These rough processes have been intensively studied as models for volatility of financial assets 
({\revarn\cite{ eleuchRougheningHeston2018,fukasawaConsistentEstimationFractional2022, gatheralVolatilityRough2018, gatheralQuadraticRoughHeston2020}}).
In this stochastic volatility models the question of estimating the roughness of the path is considered in \cite{chongStatisticalInferenceRough2024a,chongStatisticalInferenceRough2024b}, where the Hurst index $\alpha-1/2$ is estimated from observations of the price process.

%
%

In this work, we intend to address statistical questions on the estimation of the drift function of the model \eqref{Eq : Caputo form}. More formally, we consider a rough process solution of Volterra equation with {\revK kernel $K$}, \begin{equation}\label{Eq : Volterra SDE_intro}
	X^\varepsilon_t=X^\varepsilon_0+\ve \int_0^t K(t-s)a(X^\ve _s) dB_s
	+ \int_0^t K(t-s) b(X_s^\ve,\theta^\star) ds,
\end{equation}
{\revarn where $\theta^\star$ is the parameter we want to estimate.}
{\revK  We assume that the kernel $K$ is singular, with a behavior near $0$ comparable to the one of $K_0(u)=\frac{1}{\Gamma(\alpha)}
	u^{\alpha-1} \id{u>0}$.} Our setting is such that the process is close to a deterministic Volterra model as we assume a small noise asymptotic $\ve \to 0$. The objective is to estimate a parameter $\theta$ in the drift function from a discrete sampling of $X^\ve$ on some fixed interval $[0,T]$, where the sampling step $h\to0$. In the framework of semimartingale models, the small noise asymptotic has already been subject of many studies (see {\revarn \cite{
		genon-catalotParametricInferenceSmall2021,guyParametricInferenceDiscretely2014,hongweiParameterEstimationClass2010,
		kutoyantsIdentificationDynamicalSystems1994,sorensenSmalldiffusionAsymptoticsDiscretely2003b,uchidaInformationCriteriaSmall2004}}).

From a continuous observation of $X^\ve$, it is theoretically possible to compute the Caputo derivative $^CD_{0+}^T X^\ve$ and in turn consider that one observes the semimartingale process \begin{equation*}
{\revarn Z^\ve_t:=\ve \int_0^t a(X^\ve_s)dB_s+\int_0^t b(X^\ve_s,\theta^\star) ds.}
\end{equation*} 
The MLE estimator of this model is thus explicit from Girsanov's theorem.
In the context of a Volterra Ornstein-Uhlenbeck model, this method is used in \cite{zunigaVolterraProcessesApplications2021} in order to estimate the drift parameter, when $T \to \infty$. The consistency of the estimator is proved in \cite{zunigaVolterraProcessesApplications2021}. A related work is \cite{gloterDriftEstimationRough2025}, where  a non optimal trajectory fitting estimator of the drift function is constructed.

Here, we do not assume that a continuous observation of $X^\ve$  is available, which is in practice impossible,
and in turn the exact MLE {\revarn method} is inaccessible. 
The first issue solved in the paper is the 
 reconstruction of the process $Z^\ve$ from the observation $(X^\ve_{ih})_{i=0,\dots,\lfloor T/h \rfloor}$. We prove that there exists  some process $(Z^{\ve,(h)}_t)_t$, based on the observation, which satisfies for $p\ge 1$, and all $\ve\in[0,1]$,
	\begin{equation*}
	\E\left[ \abs*{Z^{\ve,(h)}_t-Z^\ve_t}^p \right] \le c(p) ( \ve^p h^{p/2} + h^{\alpha p}  ). 
\end{equation*}
Let us stress that this approximation result is certainly sharp in rate as it gives, if $\ve$ is constant, an approximation of order $h^{1/2}$ (see Theorem \ref{Th :  Lq norm Y^h Y}) which is a classical Euler rate for a semimartingale. However, the approximation is based on a sampling of $X^\ve$ which is a rough process with H\"older regularity $\alpha-1/2 <1/2$. Thus, the sampling error is reduced in our approximation procedure, mainly due to the inversion of the singular kernel $K$, which induces a smoothing effect.

Next, using this reconstruction of $Z^\ve$ as a first step, we build a contrast function which allows to estimate the drift parameter. The corresponding estimator is asymptotically Gaussian, with rate $\ve^{-1}$ and asymptotic variance given by the inverse of a Fisher information (see Theorem \ref{Thm : normality} and Remark \ref{R : Fisher}).   Let us stress that, in order to deal with the reconstruction errors and the bias induced by the contrast function, one {\revarn needs} a fast sampling condition on $h$ given by $h=o(\ve^{\frac{1}{\alpha^2}})$. Such condition is not surprising in the context of small noise asymptotic (e.g. see \cite{sorensenSmalldiffusionAsymptoticsDiscretely2003b}
 in the diffusion case). 

%

The outline of the paper is the following. In Section
\ref{S : Model}, we give the assumptions on the model.
 In Section \ref{S : Approx} we introduce the approximation procedure by discrete inversion of the Volterra kernel.
The application to the statistical estimation of the drift coefficient is presented in Section \ref{S : Stat appl}. 
In Section \ref{S : num sim}, we provide a Monte Carlo study to assess the quality of the estimator in practice. {\revarn The Appendix 
contains technical results.}


\section{Model} \label{S : Model}
We let $(\Omega,\mathcal{F}, \P)$ be a probability space, endowed with a filtration $\mathbb{F}=(\mathcal{F}_t)_{t\ge0}$ satisfying the usual conditions. We consider on this space an $r$-dimensional standard $\mathbb{F}$-Brownian motion $B$. 

On this space, we define $(X^\ve_t)_{t\in [0,T]}$ as a {\revarn$d$-}dimensional process, solution on $[0,T]$ of the stochastic differential equation
\begin{equation}\label{Eq : Volterra SDE}
X^\varepsilon_t=X^\varepsilon_0+\ve \int_0^t K(t-s)a(X^\ve _s) dB_s
+ \int_0^t K(t-s) b(X_s^\ve,\theta^\star) ds.
\end{equation}
Here, $X_0^\ve =x_0 \in \mathbb{R}^d$, 
$a : \mathbb{R}^d \to \mathbb{R}^d \otimes\mathbb{R}^r$, and 
$b : \mathbb{R}^d \times \Theta \to \mathbb{R}^d$. Here, we use the tensor product $ \mathbb{R}^d \otimes\mathbb{R}^r$ to represent the set of matrices of size $d\times r$.

The parameter $\theta^\star$ belongs to the compact and convex set $\Theta \subset \mathbb{R}^{d_{\Theta}}$. 

{\revK We assume that the kernel $K$ is of rough type. Let $\alpha \in (1/2,1)$, we set $	K_0(u)=\frac{1}{\Gamma(\alpha)}
	u^{\alpha-1} \id{u>0}$. We assume that $K$ is $\mathcal{C}^1$ on $(0,\infty)$ and its behavior can be compared to the one of $K_0$ :
\begin{equation} \label{Eq : comp K K0}
	\forall u\in(0,T],\quad \abs*{K(u)} \le c K_0(u) 
	, \quad  \abs*{K'(u)} \le c \abs*{K'_0(u)}, 
\end{equation}
for some constant $c>0$. depending on $T$.
}
%
Noteworthy, we have
\begin{equation*} 
	\sup_{t \in [0,T]} \int_0^tK(t-s)ds \le c,\quad \sup_{t \in [0,T]} \int_0^tK(t-s)^2ds \le c
\end{equation*}
for some constant $c$ {\revK depending on $T$}.

We assume $\ve\in [0,1]$ and consider the asymptotic framework $\varepsilon\to0$.

We introduce the following assumptions on the coefficients of the Volterra equation.
\begin{assumption} \label{Ass : global Lip} The functions $a$ and $b$ are continuous and
	there exists some $C>0$, such that for all $(x,x')\in \mathbb{R}^d$, $\theta \in \Theta$
	\begin{equation*}
		\abs{a(x)-a(x')} \le C \abs{x-x'},\quad
		\abs{b(x,\theta)-b(x',\theta)} \le C \abs{x-x'}.
	\end{equation*}
\end{assumption}
Let us stress that for $x\in \mathbb{R}^d$, we denote by $\abs{x}$ the Euclidean norm of $x$ given by $\abs{x}=\sqrt{x^* x}$, while for $m \in \mathbb{R}^{d}\times\mathbb{R}^r$ the quantity $\abs{m}$ is the operator norm of the $d\times r$ matrix $m$.

From {\revarn an} application of Theorem 1 in \cite{wangExistenceUniquenessSolutions2008}, we know that under \ref{Ass : global Lip}, the equation \eqref{Eq : Volterra SDE} admits a unique progressively measurable process $(X_t^\ve)_t$ as solution. 
Moreover, the solution admits finite moment of any order by Lemma 2.2. of \cite{wangExistenceUniquenessSolutions2008},
\begin{equation} \label{Eq : moment Volterra}
	\forall p \ge1, \quad \sup_{t \in [0,T]} \E[\abs{X^\ve_t}^p] \le c(p).	
\end{equation}
The constant $c(p)$ depends on $p$ and on the coefficients of the SDE \eqref{Eq : Volterra SDE} and on the kernel $K$. As $\varepsilon \in [0,1]$, the constant $c(p)$ in \eqref{Eq : moment Volterra} is uniform with respect to $\varepsilon$.

It is also known from Proposition 4.1 in \cite{richardDiscretetimeSimulationStochastic2021},
{\revK together with \eqref{Eq : comp K K0}, that} 
\begin{equation} \label{Eq : moment increment Volterra without eps}
	\E\left( \abs*{X^\ve_t-X^\ve_s}^p\right) \le c(p) \abs*{t-s}^{p (\alpha-1/2)},
\end{equation}
where the constant $c(p)$ depends on the coefficient of the stochastic differential equation. 
{\revarn As $\ve\in[0,1]$, the constant $c(p)$ can be taken independent of $\ve$. Taking into account more precisely the dependence of $\ve$, it is possible to get the sharper result 
\eqref{eq : control incre Xve}.
From \eqref{Eq : moment increment Volterra without eps}, 
we know that 
the process $X^\ve$ admits $\alpha'$-H\"older trajectories,
for any $\alpha' \in (0,\alpha-1/2)$.}
In particular, we are dealing with processes that are rougher than {\revarn Brownian motion.}
%
%
%

\section{Approximate inversion of {\revarn a} rough operator} \label{S : Approx} 
{\revarn In this section, we establish theoretical upper bounds on how the kernel $K$ can be inverted }

\subsection{Main result}

In this section we forget the asymptotic $\ve\to0$ and the underlying statistical problem. Our {\revarn goal} is to recover a semimartingale from a discrete sampling of a rough process.

Assume that $X\in\mathbb{R}^d$ is a solution of 
\begin{equation}\label{Eq : Volterra SDE_general}
	X_t=X_0+ 
	\int_0^t K(t-s)a(X_s) dB_s
	+ \int_0^t K(t-s) b(X_s) ds,
\end{equation}
where $a$ and $b$ are globally Lipschitz.
To shorten the notation, we denote by $Z$ the semimartingale $Z_t=	\int_0^ta(X_s) dB_s
+ \int_0^t b(X_s) ds$, so that  \eqref{Eq : Volterra SDE_general} {\revarn can be written as}
\begin{equation}
	\label{Eq : X as K star Z}
	X_t=X_0 + \int_0^t K(t-s)dZ_s.
\end{equation}
It is useful for {\revarn the} statistical problem to recover the values of $Z$ from the observation of $X$. 
{\revK  We introduce some useful notations and assumptions in the context of Volterra equations. We let $f \star g (t)=\int_0^t f(t-s) g(s)ds$ denote the convolution of two locally integrable functions on $(0,\infty)$. 	It is known that the kernel $K_{\revK 0}(u)=\frac{1}{\Gamma(\alpha)}
u^{\alpha-1} \id{u>0}$ admits a first kind resolvent kernel $L_{\revK 0}$ with explicit expression 
$L_{\revK 0}(u)=\frac{u^{-\alpha}}{\Gamma(1-\alpha)} \id{u>0} $, which satisfies $L_{\revK 0}\star K_{\revK 0}(t)=K_{\revK 0}\star L_{\revK 0}(t)=1$ for all $t > 0$.} {\revK We introduce the assumption that the kernel $K$ admits a first kind resolvent.
	\begin{assumption}\label{Ass : kernel resolvent}
		There exists a measurable function $L : (0,\infty) \to \mathbb{R}$, which is in $\mathbf{L}^1((0,T])$ for all $T > 0$ and such that $L\star K (t)=1$ for all $t > 0$. Moreover, we assume 
		\begin{equation}  \label{Eq : comp L L0}
			\forall u \in (0,T],~ \abs{L(u)}\le c L_0(u),
		\end{equation}
		for some constant $c$ depending on $T$.
	\end{assumption}
Conditions on $K$ ensuring the existence of the first kind resolvent $L$ may be found in \cite{gripenbergVolterraEquationsFirst1980}. 
As $L$ is an inverse of $K$, the upper bound constraints \eqref{Eq : comp L L0} and \eqref{Eq : comp K K0} are opposite. We prove in Proposition
\ref{Prop : suff cond for A2}, that if $K$ is some perturbation of $K_0$, then \eqref{Ass : kernel resolvent} is valid. 
}

Now, it is possible to recover the path of $Z$ from the observation of $X$ by using a stochastic convolution as done in
\cite{abijaberAffineVolterraProcesses2019}. 
We repeat here the idea, which is the starting point of our approximation scheme.
First, we define the stochastic convolution
$ J \star dM(t)=\int_0^t
J(t-s) dM_s$
 for $M$ a continuous semimartingale process with $\bracket*{M}{M}_t=\int_0^t a_s ds$, where $a$ is a locally bounded process $s \mapsto a_s$, 
and $J : [0,\infty) \to  \mathbb{R}$ any locally  ${\revarn\mathbf{L}}^2$ kernel. 
With this notation,  we have $X_t=X_0+K\star dZ(t)$. We define $Y_t:=\int_0^tL(t-s)(X_s-X_0)ds=L\star\overline{X}(t)$ with $\overline{X}_s=X_s-X_0$. Now,
\begin{align}\nonumber	
	Y_t&=L \star \overline{X}(t)\\	
	\nonumber
& =L \star \big( K \star dZ\big)(t)\\
\label{Eq : using stoch associativity}
& =\big(L \star  K \big) \star dZ(t)\\
\nonumber
& = \one_{\{\cdot >0\}} \star dZ (t)=\int_0^t \one_{u >0} dZ_u=Z_t-Z_0=Z_t.
\end{align}
To get the line \eqref{Eq : using stoch associativity}, we use the associativity of the stochastic convolution as shown in Lemma 2.1 of
\cite{abijaberAffineVolterraProcesses2019}.

This inversion of the kernel allowing to recover $Z$ from the observation of $X$ is unfeasible if $X$ is not continuously observed. However, we now introduce an approximate inversion based on a discrete observation of $X$. 

Let {\revarn $0<h<1$} be the sampling step and we denote $\varphi_h(t)=h \times \lfloor
\frac{t}{h} \rfloor$ which is such that $\varphi_h(t) \le t < \varphi_h(t)+h$ and $\varphi_h(t) \in h \mathbb{N}$.

The discrete observation of the process $(X_t)_{t \in [0,T]}$ on the grid with step $h$ is equivalent to the continuous observation of the process $(X^{(h)}_t)_{t \in [0,T]}$ with $X^{(h)}_t=X_{\varphi_h(t)} $.
We define 
\begin{equation} \label{Eq : def Y^h}
	Z^{(h)}_t=\int_0^t L(t-s) (X^{(h)}_s -X^{(h)}_0)ds.
\end{equation}
We expect that for small $h$, the process $Z^{(h)}$ is close to $Z_t$.
The main result about the quality of approximation is the following.
\begin{theorem}\label{Th :  Lq norm Y^h Y}
{\revK Assume \ref{Ass : global Lip} and \ref{Ass : kernel resolvent}.}
	Then, for any $p\ge 1$, {\revarn there exists $c(p)>0$ such that for all $t\in[0,T]$,}
	 \begin{equation} \label{Eq : Lp Y - Z in statement}
	\norm{Z^{(h)}_t-Z_t}_{\mathbf{L}^p} \le
		\ c(p) h^{1/2}  \sup_{s\in[0,T]} \norm{a(X_s)}_{\mathbf{L}^p} 
		+
		c(p) h^{\alpha}  \sup_{s\in[0,T]} \norm{b(X_s)}_{\mathbf{L}^p} .		
	\end{equation}
	In particular, it gives
	\begin{equation}\label{Eq : Lp Y - Z in statement particular}
		\sup_{t \in [0,T]} \norm{	Z^{(h)}_t - Z_t }_{\mathbf{L}^p}
		\le c(p,a,b) h^{1/2}
	\end{equation}
{\revarn  for some constant $ c(p,a,b)$ depending on $p$ and the coefficients of the SDE.}
\end{theorem} 
{\revarn Remark that the constant $c(p)$ in \eqref{Eq : Lp Y - Z in statement} depends only on $p$ and {\revK the kernel K}.} This theorem will be a consequence of the representation of $Z^{(h)}$ using convolution kernels, with explicit control on this kernel. 
\begin{remark} \label{rem : certainly rate optimal} 
Let us stress that the rough process $X$ is $\alpha'$-H{\"o}lder for
 $\alpha'<\alpha-1/2<1/2$ and {\revarn may therefore} be very irregular for $\alpha$ close to $1/2$. In this case the quality of {\revarn the} approximation $X$ by $X^{(h)}$ is poor, with an error of magnitude $h^{\alpha-1/2}$, {\revarn which is} thus much larger than $h^{1/2}$.  By Theorem \ref{Th :  Lq norm Y^h Y}, we see that the magnitude of the error is reduced in the inversion process as the error of approximation of $Z_t$ by $Z^{(h)}_t$ is of magnitude $h^{1/2}$. As the process $Z$ is a Brownian semimartingale, a direct sampling with step $h$ would produce an error of magnitude $h^{1/2}$. For this reason, we conjecture that the $h^{1/2}$ rate in Theorem \ref{Th :  Lq norm Y^h Y} might be {\revarn sharp,} as  we do not expect the approximation of $Z$ based on the sampling of $X$ to be {\revarn more accurate}
 than the approximation {\revarn obtained from} the sampling of $Z$ itself.
\end{remark}

\subsection{Representation as convolution}

First, we have to generalize the convolution notation to {\revarn kernels with  two variables.} 
{\revarn 
Let $\widetilde{K} : (0,T]^2 \mapsto \mathbb{R} $ and $\widetilde{L} : (0,T]^2 \mapsto \mathbb{R} $ be measurable functions such that
$\sup_{u\in(0,T]} \norm*{\widetilde{K}(u,\cdot)}_{\mathbf{L}^2((0,T])}<\infty$, and $\sup_{u\in(0,T]} \norm*{\widetilde{L}(u,\cdot)}_{\mathbf{L}^1((0,T])}<\infty$.
Then,  
 we can define the $\widetilde{L} \star \widetilde{K} (t,s) = \int_0^T \widetilde{L}(t,u)  \widetilde{K}(u,s) du$ for $(s,t)\in(0,T]^2$ and
 check  $\sup_{t\in(0,T]} \norm*{\widetilde{L} \star \widetilde{K}(t,\cdot)}_{\mathbf{L}^2((0,T])}<\infty$.
If the kernels are one sided, which means that $\widetilde{L}(t,s)=0$ and $\widetilde{K}(t,s)=0$ for $s>t$, then $\widetilde{L} \star \widetilde{K}$ is one sided as well. 
Also if $\widetilde{K}(s,t)=\widetilde{K}^1(t-s)\one_{0<s<t}$ and $\widetilde{L}(s,t)=\widetilde{L}^1(t-s)\one_{0<s<t}$ for $\widetilde{K}^1 \in \mathbf{L}^2((0,T])$ and $\widetilde{L}^1 \in   \mathbf{L}^1((0,T])$, then,
$\widetilde{L} \star \widetilde{K} (t,s) = \widetilde{L}^1\star \widetilde{K}^1 (t-s)$, where $\widetilde{L}^1\star \widetilde{K}^1$ is the usual convolution on $[0,\infty)$. In particular $\widetilde{L} \star \widetilde{K} (t,0) = \widetilde{L}^1\star \widetilde{K}^1 (t)$ relates the two-variable convolution with the classical one. 
We can also consider 
$\widetilde{L}^1 \star \widetilde{K} (t,s) = \widetilde{L} \star \widetilde{K} (t,s) =\int_0^t	 \widetilde{L}^1(t-u)\widetilde{K}(u,s)du$, and 
define if $M$ is a semi martingale,
$\widetilde{K} \star dM (t,s)= \int_s^T \widetilde{K}(t,u) dM_u$, which is  $\int_s^t \widetilde{K}(t,u) dM_u$ if the kernel is one sided.
}

Let us set 
\begin{align}\label{Eq : def Kh}
	K^{(h)}(t,s)&=K(\varphi_h(t)-s) \one_{0\le s<\varphi_h(t)}, \\
	\label{Eq : def g^h}
	g^{(h)}(t,s)&=L\star K^{(h)} (t,s)=\int_s^t L(t-u) K^{(h)}(u,s) du,
\end{align}
where we recall {\revK that $L$ is the first kind resolvent defined in \ref{Ass : kernel resolvent}.}
The following proposition shows how to write $X^{(h)}$ using a bi-indexed kernel, and generalizes in that context the stochastic associativity argument appearing in \eqref{Eq : using stoch associativity}.
\begin{proposition} {\revK Assume \ref{Ass : global Lip} and \ref{Ass : kernel resolvent}.}
Let us denote $\overline{X}^{(h)}_t=X_t^{(h)}-X_0^{(h)}$. Then,  we have for all $t \in [0,T]$,
	\begin{align} \label{Eq : X^h as kernel}
		\overline{X}^{(h)}_t&=K^{(h)}\star d Z (t,0),\\
		\label{Eq : Y^h as kernel}
		{Z}^{(h)}_t&=g^{(h)}\star d Z (t,0)=\int_0^t g^{(h)}(t,s) dZ_s.
	\end{align}
\end{proposition}
	\begin{proof}
{\revarn First, we prove \eqref{Eq : X^h as kernel}, 
	which} is a consequence of the notations introduced earlier. Indeed, recalling \eqref{Eq : X as K star Z}, we can write $\overline {X}^{(h)}_t=X_{\varphi_h(t)}-X_0=\int_0^{\varphi_h(t)} K(\varphi_h(t)-s) d Z_s= \int_0^t K^{(h)}(t,s)dZ_s=K^{(h)}\star dZ (t,0)$, where we used \eqref{Eq : def Kh}.
		
		To get \eqref{Eq : Y^h as kernel}, we write
		\begin{align}\nonumber
			Z^{(h)}_t&=\int_0^t L(t-s) (X^{(h)}_s -X^{(h)}_0)ds, \quad	\text{{\revarn recalling} \eqref{Eq : def Y^h},}
			\\\nonumber
			&=\int_0^t L(t-s) \left(\int_0^s K^{(h)}(s,u) dZ_u\right)ds, \quad	\text{by \eqref{Eq : X^h as kernel},}
			\\\nonumber
	&=\int_0^t L(t-s) \left(\int_0^t K^{(h)}(s,u)\one_{u <s} dZ_u\right)ds, 	\quad \text{as the kernel $K^{(h)}$ is one sided,}
			\\\nonumber
		&=\int_0^t \left(\int_0^t  L(t-s) K^{(h)}(s,u)\one_{u <s} ds\right)dZ_u, 
		\\ \nonumber &\quad\quad\quad\quad\quad\quad\text{using stochastic Fubini theorem (see Theorem 65 in
			\cite{protterStochasticIntegrationDifferential2005}),}
			\\ \nonumber
			&=\int_0^t L\star K^{(h)}(t,u) dZ_u=\int_0^t g^{(h)}(t,u)dZ_u, \quad \text{ by \eqref{Eq : def g^h}}.
			\end{align}
		Remark that the application of the stochastic Fubini theorem (see Theorem 65 in
		\cite{protterStochasticIntegrationDifferential2005}) is subject to an integrability condition on the integrand. A sufficient condition is
		$\int_0^{\revarn t} \int_{{\revarn 0}}^t L({\revarn t-s})K^{(h)}(s,u)^2 ds du < \infty$. This condition can be checked using the Fubini-Tonelli theorem and the fact {\revK $L$ is in $\mathbf{L}^{1}((0,T])$ together with 
		$\int_0^s K^{(h)}(s,u)^2du=\int_0^{\varphi_h(s)}K(\varphi_h(s)-u)^2 du\le 
		 c \int_0^{\varphi_h(s)}K_0(\varphi_h(s)-u)^2 du \le
		c\int_0^T K_0(t)^2 dt <\infty$, where we used \eqref{Eq : comp K K0}.}	
	\end{proof}
{\revarn Using} \eqref{Eq : Y^h as kernel}, we see that
\begin{equation*}
	Z^{(h)}_t-Z_t=\int_0^t (g^{(h)}(t,u)-1) dZ_u.
\end{equation*} 
It leads us to compare $g^{(h)}(t,u)=L\star K^{(h)}(t,u)$ with $1=L\star K (t)$.
{\revarn An} additional notation is necessary in order to write down our assessment on $g^{(h)}(t,u)-1$. 
For $u \in [0,T]$, we denote $\chi_h(u)=u-\varphi_h(u) \in [0,h)$.
\begin{proposition} \label{Prop : maj g - 1}  
{\revK Assume \ref{Ass : kernel resolvent}.}	There exists {\revarn a constant $C>0$ depending only on the kernel $K$,} such that for all $0\le u < t \le T$,
	\begin{equation*}
		\abs*{g^{(h)}(t,u)-1} \le C \left[\left(\frac{h}{t-u}\right)^\alpha \wedge 1 \right]  + 
		Ch {\revK K_0}(h-\chi_h(u)) \left[\left(\frac{1}{t-u}\right)^\alpha \wedge \left(\frac{1}{h}\right)^\alpha\right].
	\end{equation*} 
\end{proposition}
\begin{proof}
	To have homogeneous notations, we set $K(t,u)=K(t-u)\one_{0\le u <t}$, and it gives, recalling \eqref{Eq : def g^h}
	\begin{equation*}
	\Xi^{(h)}(t,u):= g^{(h)}(t,u)-1=L\star(K^{(h)}-K)(t,u)=\int_u^t L(t-v)( K^{(h)}(v,u)-K(v,u) )dv, 
		 	\end{equation*} 
for $0 \le u <t$, and where we used $L\star K(t,u)=\one_{0\le u < t} $.
We set $\xi=(v-u)/(t-u)$ and deduce
	\begin{equation*}
	\Xi^{(h)}(t,u)= (t-u) \int_0^1 L((1-\xi) (t-u))\big(K^{(h)}(\xi(t-u)+u,u)-K(\xi(t-u)+u,u) \big)d\xi.
\end{equation*} 
We use that {\revK $L(s)\le c L_0(s) \le C s^{-\alpha}$ to write, with a constant $C$ which may change from line to line,} 
	\begin{align*}
	\abs*{\Xi^{(h)}(t,u)}&\le C (t-u)^{1-\alpha} \int_0^1 (1-\xi)^{-\alpha} \abs[\Big]{K^{(h)}(\xi(t-u)+u,u)-K(\xi(t-u)+u,u)}d\xi
	\\
	& = C \sum_{i=1}^3 e_i^{(h)}(t,u),
\end{align*} 
where
\begin{align*}
	e_1^{(h)}(t,u)&:=\begin{multlined}[t]
	(t-u)^{1-\alpha} \int_0^{\frac{h-\chi_h(u)}{t-u} \wedge 1}  (1-\xi)^{-\alpha} \abs[\Big]{K^{(h)}(\xi(t-u)+u,u)-\\K(\xi(t-u)+u,u)}d\xi,
	\end{multlined} 
	\\
	e_2^{(h)}(t,u)&:= 
	\begin{multlined}[t] (t-u)^{1-\alpha} \int_{\frac{h-\chi_h(u)}{t-u} \wedge 1}^{\frac{3h-\chi_h(u)}{t-u} \wedge 1}  (1-\xi)^{-\alpha}
	\abs[\Big]{K^{(h)}(\xi(t-u)+u,u)-\\K(\xi(t-u)+u,u)}d\xi,
\end{multlined}
	\\
	e_3^{(h)}(t,u)&:= 
	\begin{multlined}[t] (t-u)^{1-\alpha} \int_{\frac{3h-\chi_h(u)}{t-u} \wedge 1}^1  (1-\xi)^{-\alpha} \abs[\Big]{K^{(h)}(\xi(t-u)+u,u)-\\K(\xi(t-u)+u,u)}d\xi.
\end{multlined}  
\end{align*}

We focus on $e_1^{(h)}(t,u)$. For $\xi \in [0, \frac{h-\chi_h(u)}{t-u})$, we have 	$\xi(t-u)+u \in [u,h-\chi_h(u)+u)=[u,\varphi_h(u)+h) $ using $\chi_h(u)=u-\varphi_h(u)$. We deduce that $\varphi_h(\xi(t-u)+u) =\varphi_h(u)\le u$ and consequently $K^{(h)}(\xi(t-u)+u,u)=K(\varphi_h(\xi(t-u)+u)-u)\one_{0\le u<\varphi_h(\xi(t-u)+u)}=0$. Consequently,
\begin{align}\nonumber
	e_1^{(h)}(t,u)&=  (t-u)^{1-\alpha} \int_0^{\frac{h-\chi_h(u)}{t-u} \wedge 1}  (1-\xi)^{-\alpha} \abs[\Big]{K(\xi(t-u)+u,u)}d\xi,
	\\ \label{Eq : control e1h}
	& {\revK \le  C} \int_0^{\frac{h-\chi_h(u)}{t-u} \wedge 1}  (1-\xi)^{-\alpha} \xi^{\alpha-1}d\xi,
\end{align}
where we used {\revK $\abs*{K(\xi(t-u)+u,u)}=\abs*{K(\xi(t-u))}
	\le K_0(\xi(t-u)) \le C \xi^{\alpha-1}(t-u)^{\alpha-1}$.} 
	 Since $\int_0^1 (1-\xi)^{-\alpha}\xi^{1-\alpha}d\xi <\infty$, we get $e_1^{(h)}(t,u) \le C$.  Moreover, in the case where $t-u \ge 2h$, we have $\frac{h-\chi_h(u)}{t-u}\le \frac{h}{2h}=\frac{1}{2}$ and we deduce from \eqref{Eq : control e1h}
\begin{equation*}
	e_1^{(h)}(t,u)\le {\revK C} 
	\int_0^{\frac{h-\chi_h(u)}{t-u}}   {\revK 2^\alpha \xi^{\alpha-1} }d\xi 
	 \le C \left(\frac{h-\chi_h(u)}{t-u}\right)^\alpha \le  C \left(\frac{h}{t-u}\right)^\alpha.
\end{equation*}
{\revarn Combining} this {\revarn upper bound} with the boundedness of $e_1^{(h)}(t,u)$, we deduce
\begin{equation}\label{Eq : upper e1}
		e_1^{(h)}(t,u)\le C \left[\left(\frac{h}{t-u}\right)^\alpha \wedge 1 \right] .
\end{equation}

We now study $e_2^{(h)}(t,u)$. Remark that for $\xi  \in [\frac{kh-\chi_h(u)}{t-u},\frac{(k+1)h-\chi_h(u)}{t-u})$ with $k\ge 1$ some integer, we have
$\xi(t-u)+u \in  [kh-\chi_h(u)+u   , kh+h-\chi_h(u)+u)=[kh + \varphi_h(u)   , kh+h + \varphi_h(u))$. We deduce $\varphi_h(\xi(t-u)+u)=kh+\varphi_h(u)$. Thus,
$K^{(h)}(\xi(t-u)+u,u)=K(\varphi_h(\xi(t-u)+u)-u)=K(kh+\varphi_h(u)-u)=K(kh-\chi_h(u))$, where we
used $u=\varphi_h(u)+\chi_h(u)$. {\revK From \eqref{Eq : comp K K0}, this yields $\abs*{K^{(h)}(\xi(t-u)+u,u)}\le C K_0(kh-\chi_h(u))$.}
As a consequence, if $\xi  \in [\frac{h-\chi_h(u)}{t-u},\frac{3h-\chi_h(u)}{t-u})$, 
we deduce {\revK $\abs*{ K^{(h)}(\xi(t-u)+u,u)} \le C K_0(h-\chi_h(u)) $} using that $s\mapsto K_{{\revK 0}}(s)$ is a decreasing function on $(0,\infty)$.
Also, for $\xi  \in [\frac{h-\chi_h(u)}{t-u},\frac{3h-\chi_h(u)}{t-u})$, we have 
{\revK $\abs*{ K(\xi(t-u)+u,u) } = \abs*{ K(\xi(t-u)+u-u) } \le 
	C K_0(\xi(t-u))$, where we used 
		\eqref{Eq : comp K K0}. Since $s\mapsto K_{{\revK 0}}(s)$ is decreasing on $(0,\infty)$ and $\xi(t-u) \ge h-\chi_h(u)$, we deduce
$\abs*{K(\xi(t-u)+u,u)} \le C K_{{\revK 0}}(h-\chi_h(u))$.}
From this {\revarn and the definition of $e_2^{(h)}(t,u)$},
we can write
\begin{equation}\label{Eq : upper e2 inter}
	e_2^{(h)}(t,u)\le  {\revK C}(t-u)^{1-\alpha} K_{\revK 0{}}(h-\chi_h(u)) \int_{\frac{h-\chi_h(u)}{t-u} \wedge 1}^{\frac{3h-\chi_h(u)}{t-u} \wedge 1}  (1-\xi)^{-\alpha}d\xi.
\end{equation}
If $t-u \ge 4h$, we have $\frac{3h-\chi_h(u)}{t-u} \le 3/4$ and we can upper bound 
\begin{equation*}
	\int_{\frac{h-\chi_h(u)}{t-u} \wedge 1}^{\frac{3h-\chi_h(u)}{t-u} \wedge 1} (1-\xi)^{-\alpha}  d\xi\le 
		\int_{\frac{h-\chi_h(u)}{t-u}}^{\frac{3h-\chi_h(u)}{t-u} } (1/4)^{-\alpha} d\xi \le C \frac{h}{t-u}.
	\end{equation*}
Moreover, we have $	\int_{\frac{h-\chi_h(u)}{t-u} \wedge 1}^{\frac{3h-\chi_h(u)}{t-u} \wedge 1} (1-\xi)^{-\alpha}  d\xi\le  C$ for all $(t,u)$. Inserting these controls in \eqref{Eq : upper e2 inter} gives
\begin{multline}\label{Eq : upper e2}
	e_2^{(h)}(t,u)\le  C (t-u)^{1-\alpha} K_{{\revK 0}}(h-\chi_h(u)) 
	\left[\frac{h}{t-u} \wedge 1 \right]
	\\
	\le  C K_{{\revK 0}}(h-\chi_h(u))h 
	\left[\frac{1}{(t-u)^\alpha} \wedge \frac{1}{h^{\alpha}}  \right],
	\end{multline}
where to get the second line we discussed according to the minimum between $1$ and $h/(t-u)$.

It remains to study $e_3^{(h)}(t,u)$.   We use $\varphi_h(\xi(t-u)+u) \ge\xi(t-u)+u-h$ to deduce
$   \varphi_h(\xi(t-u)+u)-u \ge \xi(t-u)-h=\frac{\xi}{2}(t-u) + \frac{\xi}{2}(t-u)-h$. Now, in the situation
$\xi \ge \frac{3h-\chi_h(u)}{t-u}$, we have $\frac{\xi}{2}(t-u) \ge \frac{1}{2}(3h-\chi_h(u))\ge 2h/2=h$. It entails,  $\varphi_h(\xi(t-u)+u)-u \ge\frac{\xi}{2}(t-u) + \frac{\xi}{2}(t-u)-h \ge \frac{\xi(t-u)}{2}$.

From the definition of $K^{(h)}(\cdot,\cdot)$ and $K(\cdot,\cdot)$, we have
\begin{multline} \label{Eq : e3 to bound}
	e_3^{(h)}(t,u)=  (t-u)^{1-\alpha} \int_{\frac{3h-\chi_h(u)}{t-u} \wedge 1}^1  (1-\xi)^{-\alpha} \abs[\big]{K(\varphi_h(\xi(t-u)+u)-u)\\-K(\xi(t-u))}d\xi.
\end{multline}
In the case $t-u \ge 4h$, we use {\revarn the mean value theorem to obtain} 
$\abs{K(\varphi_h(\xi(t-u)+u)-u)-K(\xi(t-u))} \le \abs{\varphi_h(\xi(t-u)+u)-u - \xi(t-u)} \times
\sup_{s \in [\varphi_h(\xi(t-u)+u)-u {\revarn )},\xi(t-u)] }\abs{K'(s)}$. As we have seen that $\varphi_h(\xi(t-u)+u)-u \ge \frac{\xi}{2}(t-u)$, and using that $\abs{\varphi_{{\revarn h}}(s)-s}\le h$ for any $s$, it implies
\begin{multline*}
	\abs*{K(\varphi_h(\xi(t-u)+u)-u)-K(\xi(t-u))} \le h \sup_{s \in [\frac{\xi}{2}(t-u), 1] }\abs{K'(s)}
\\	{\revK \le C h \sup_{s \in [\frac{\xi}{2}(t-u), 1] }\abs{K'_0(s)}}
	\le C h \left({\revarn \xi(t-u)}\right)^{\alpha-2},
\end{multline*}
{\revK where we used \eqref{Eq : comp K K0}.}
Plugging in \eqref{Eq : e3 to bound}, we deduce
\begin{equation*}
		e_3^{(h)}(t,u) \le C h(t-u)^{-1} \int_{\frac{3h-\chi_h(u)}{t-u} \wedge 1}^1  (1-\xi)^{-\alpha} \xi^{\alpha-2} d\xi.
\end{equation*} 
In the case $t-u \ge 4h$, the integral in the equation above is bounded by $C \left(\frac{3h-\chi_h(u)}{t-u}\right)^{\alpha-1} \le C
\left(\frac{h}{t-u}\right)^{\alpha-1} $, we deduce that
\begin{equation*}
	e_3^{(h)}(t,u) \le C \left(\frac{h}{t-u}\right)^{\alpha}.
\end{equation*} 
In the case $t-u < 4h$, we use \eqref{Eq : e3 to bound} and recall that in the integrand  $\varphi_h(\xi(t-u)+u)-u  \ge \frac{\xi(t-u)}{2}$. 
{\revK As $\abs*{K}\le C\abs{K_0}$ and the function $s \mapsto K_0(s)$ is decreasing on $(0,\infty)$, it entails}
\begin{align*} 
	e_3^{(h)}(t,u)& \le  {\revK C} (t-u)^{1-\alpha} \int_{\frac{3h-\chi_h(u)}{t-u} \wedge 1}^1  (1-\xi)^{-\alpha}
	 \abs*{K_{{\revK 0}}\left(\frac{\xi}{2}(t-u)\right)}d\xi
	\\
	&\le {\revK C} \int_{\frac{3h-\chi_h(u)}{t-u} \wedge 1}^1  (1-\xi)^{-\alpha} \xi^{\alpha-1}d\xi \le C.
\end{align*}
Gathering the cases $t-u < 4h$ and $t-u \ge 4h$, we deduce
\begin{equation} \label{Eq : upper e3}
	e_3^{(h)}(t,u) \le C \left[ \left(\frac{h}{t-u}\right)^{\alpha} \wedge 1 \right]
\end{equation}
The proposition is a consequence of \eqref{Eq : upper e1}, \eqref{Eq : upper e2} and \eqref{Eq : upper e3}.
\end{proof}
\begin{lemma} \label{lem : upper bound L1 L2 g minus 1}
	There exist {\revarn $C_1, ~C_2>0$} such that for all $0\le t \le T$, $0\le  h \le 1$,
	\begin{align} \label{Eq : int g 1 h alpha}
		&\int_0^t \abs*{g^{(h)}(t,u)-1}du \le {\revarn C_1} h^\alpha,
		\\ \label{Eq : int g 2 h half}
		&\int_0^t \abs*{g^{(h)}(t,u)-1}^2 du \le {\revarn C_2} h.
	\end{align}
\end{lemma}	
\begin{proof}
	We start with the proof of \eqref{Eq : int g 1 h alpha}. We use Proposition \ref{Prop : maj g - 1}.
	It leads us to consider for $0 \le t \le T$, $0\le h\le 1$,
	\begin{multline}\label{Eq : int g 1 h alpha simple}
		\int_0^t  \left[\left(\frac{h}{t-u}\right)^\alpha \wedge 1 \right]  
		du \le \int_0^{(t-h)\vee 0} \left(\frac{h}{t-u}\right)^\alpha du + \int_{(t-h)\vee 0}^t 1 du
	\\	{\revarn \le C h^{\alpha}t^{1-\alpha} + C h\le C h^\alpha. } 
	\end{multline}
	
			For $t \ge 2h$, we write
\begin{multline*}
		h\int_0^t  K_{{\revK 0}}(h-\chi_h(u)) \left[\left(\frac{1}{t-u}\right)^\alpha \wedge \left(\frac{1}{h}\right)^\alpha\right] du
		\le \int_0^{\varphi_h(t)-h} \frac{hK_{{\revK 0}}(h-\chi_h(u))}{(t-u)^\alpha} du\\
		{\revarn +} h^{1-\alpha}\int_{\varphi_h(t)-h}^{\varphi_h(t)+h} K_{{\revK 0}}(h-\chi_h(u))du.
\end{multline*}
Let us denote by $i_h \ge 2$ the integer such that $\varphi_h(t)=h i_h$. We split the first integral in the right hand-side of the above inequality as
\begin{equation*}
	\sum_{j=0}^{i_h-2} \int_{jh}^{(j+1)h} \frac{hK_{{\revK 0}}(h-\chi_h(u))}{(t-u)^\alpha} du \le 
		\sum_{j=0}^{i_h-2} \frac{h}{(i_h h-(j+1)h)^\alpha} \int_{jh}^{(j+1)h} K_{{\revK 0}}(h-\chi_h(u)) du . 
\end{equation*}
For any $j \ge 0$, $ \int_{jh}^{(j+1)h} K_{{\revK 0}}(h-\chi_h(u)) du=
\int_{jh}^{(j+1)h} K_{{\revK 0}}(h-(u-\varphi_h(u)))du=
\int_{jh}^{(j+1)h} K_{{\revK 0}}(h-(u-jh))=\int_0^h K_{{\revK 0}}(u) du \le C h^\alpha$. We deduce
\begin{multline}\label{Eq : main L1 g}
	\sum_{j=0}^{i_h-2} \int_{jh}^{(j+1)h} \frac{hK_{{\revK 0}}(h-\chi_h(u))}{(t-u)^\alpha} du \le 
	\sum_{j=0}^{i_h-2}  \frac{h^{1+\alpha}}{(i_h h-(j+1)h)^\alpha} \\   
	=h\sum_{j=0}^{i_h-2}  \frac{1}{(i_h -(j+1))^\alpha} =
		h\sum_{j=1}^{i_h-1}  \frac{1}{j^\alpha} \le C h(i_h)^{1-\alpha} 
		\\= C (h i_h)^{1-\alpha} h^\alpha		
		\le C t^{1-\alpha} h^{\alpha} \le C h^\alpha.
\end{multline}	
We also have $h^{1-\alpha} \int_{\varphi_h(t)-h}^{\varphi_h(t)+h} K_{{\revK 0}}(h-\chi_h(u)) du=2h^{1-\alpha} \int_0^hK_{{\revK 0}}(u)du \le C h$. 
Together with \eqref{Eq : main L1 g}, it gives
 $	h\int_0^t  K_{{\revK 0}}(h-\chi_h(u)) \left[\left(\frac{1}{t-u}\right)^\alpha \wedge \left(\frac{1}{h}\right)^\alpha\right] 
du \le C(h+h^\alpha) \le C h^\alpha$ for $t \ge 2h$. Now, \eqref{Eq : int g 1 h alpha} follows in this case, using \eqref{Eq : int g 1 h alpha simple} and recalling Proposition \ref{Prop : maj g - 1}.

 When $t \le 2h$, we write 
$	h\int_0^t  K_{{\revK 0}}(h-\chi_h(u)) \left[\left(\frac{1}{t-u}\right)^\alpha \wedge \left(\frac{1}{h}\right)^\alpha\right] du \le 
h^{1-\alpha}\int_{0}^{\varphi_h(t)+h} K_{{\revK 0}}(h-\chi_h(u))du \le C h \le Ch^{\alpha}$ and  \eqref{Eq : int g 1 h alpha} follows,  as well.

We now prove \eqref{Eq : int g 2 h half}. Using Proposition \ref{Prop : maj g - 1}, it is a consequence of \eqref{Eq : int g2 h half simple} and \eqref{Eq : int g2 h half comp} below.
We have the simple upper bound, for $0 \le t \le T$, $0\le h\le 1$,
\begin{multline}\label{Eq : int g2 h half simple}
	\int_0^t  \left[\left(\frac{h}{t-u}\right)^\alpha \wedge 1 \right]^2  
	du \le \int_0^{(t-h)\vee 0} \left(\frac{h}{t-u}\right)^{2\alpha} du + \int_{(t-h)\vee 0}^t 1 du
	\\	\le C h^{2\alpha}h^{1-2\alpha} + h \le C  h.
\end{multline}

		For $t \ge 2h$, we write
\begin{align*}
&	h^2\int_0^t  K_{{\revK 0}}(h-\chi_h(u))^2 \left[\left(\frac{1}{t-u}\right)^{\alpha} \wedge \left(\frac{1}{h}\right)^{\alpha}\right]^2 du
\\	&\le C\int_0^{\varphi_h(t)-h} \frac{h^2K_{{\revK 0}}(h-\chi_h(u))^2}{(t-u)^{2\alpha}} du+ h^{2-2\alpha}\int_{\varphi_h(t)-h}^{\varphi_h(t)+h} K_{{\revK 0}}(h-\chi_h(u))^2du.
\\	&\le  C	\sum_{j=0}^{i_h-2} \int_{jh}^{(j+1)h} \frac{h^2K_{{\revK 0}}(h-\chi_h(u))^2}{(i_h h-(j+1)h)^{2\alpha}} du 
+\sum_{i=i_h-1}^{i_h+1} h^{2-2\alpha} \int_{jh}^{(j+1)h} K_{{\revK 0}}(h-\chi_h(u))^2du
\\
& \le  C	h^{2-2\alpha } \int_{0}^{h} K_{{\revK 0}}(u)^2 du \sum_{j=0}^{i_h-2} \frac{1}{(i_h -(j+1))^{2\alpha}} 
+ h^{2-2\alpha} 2 \int_{0}^{h} K_{{\revK 0}}(u)^2 du.
\end{align*}
Using that $\int_0^h K_{{\revK 0}}(u)^2 du \le C h^{2\alpha-1}$ and $\sum_{j=1}^\infty \frac{1}{j^{2\alpha}}<\infty$, we deduce
\begin{equation}\label{Eq : int g2 h half comp}
	h^2\int_0^t  K_{{\revK 0}}(h-\chi_h(u))^2 \left[\left(\frac{1}{t-u}\right)^{\alpha} \wedge \left(\frac{1}{h}\right)^{\alpha}\right]^2 du \le Ch.
\end{equation}

If $t \le 2h$, the equation \eqref{Eq : int g2 h half comp} is also valid, using
\begin{multline*}
		h^2\int_0^t  K_{{\revK 0}}(h-\chi_h(u))^2 \left[\left(\frac{1}{t-u}\right)^{\alpha} \wedge \left(\frac{1}{h}\right)^{\alpha}\right]^2 du \\\le 
h^{2-2\alpha} \int_0^{2h} K_{{\revK 0}}(h-\chi_h(u))^2 du \le Ch.
\end{multline*} 
The lemma is proved.
\end{proof}
{\revarn 
	\begin{remark}
	The constant $C_1$ in Lemma \ref{lem : upper bound L1 L2 g minus 1} depends on the kernel $K$ and on the time horizon $T$, with $C_1=O(T^{1-\alpha})$ for $T\to\infty$. On the other hand, the constant $C_2$ depends only on the kernel. 
\end{remark}}
\subsection{Proof of Theorem \ref{Th :  Lq norm Y^h Y}}
First, remark that is it sufficient to prove the result for $p\ge 2$.
We write, using \eqref{Eq : Y^h as kernel} and the notation
$\xi^{(h)}(t,s)=g^{(h)}(t,s)-1$,
\begin{align*}
Z^{(h)}_t-Z_t&=\int_0^t g^{(h)}(t,s) dZ_s -Z_t=\int_0^t \xi^{(h)}(t,s) dZ_s \\
&=
\int_0^t \xi^{(h)}(t,s)  a(X_s) dB_s + \int_0^t \xi^{(h)}(t,s)  b(X_s) ds. 
\end{align*}
Thus, we have using {\revarn Burkholder}-Davis-Gundy inequality in the second line,
\begin{align*}
		\E\left[\abs*{Z^{(h)}_t-Z_t}^p\right]&
\begin{multlined}[t]
	\le 2^{p-1}  	\E\left[\abs*{\int_0^t \xi^{(h)}(t,s)  a(X_s) dB_s}^p\right]  \\ {\revarn +} 
	2^{p-1}  	\E\left[\abs*{\int_0^t \xi^{(h)}(t,s)  b(X_s) ds}^p\right] 
\end{multlined}
\\&
\begin{multlined}
	\le 2^{p-1} c(p) 	\E\left[\abs*{\int_0^t \xi^{(h)}(t,s)^2  \abs*{a(X_s)}^2 ds}^{p/2}\right]  \\ {\revarn +} 
	2^{p-1}  	\E\left[\abs*{\int_0^t \xi^{(h)}(t,s)  b(X_s) ds}^p\right] .
\end{multlined}
\end{align*}
Now, we apply Jensen's inequality, as $p\ge2$, to deduce,
\begin{align*}
	\E\left[\abs*{Z^{(h)}_t-Z_t}^p\right]&
	\begin{multlined}[t]
		\le c(p)  	\E\left[ \left( \int_0^t \xi^{(h)}(t,s)^2 ds \right)^{p/2-1} 
		\int_0^t \xi^{(h)}(t,s)^2   \abs{a(X_s)}^p ds \right]  \\
		{\revarn +}
		c(p)	\E\left[\left(\int_0^t \abs{\xi^{(h)}(t,s)} ds \right)^{p-1} \int_0^t \abs{\xi^{(h)}(t,s)} \abs{b(X_s)}^p ds \right] 
	\end{multlined}
	\\&
	\begin{multlined}
		\le c(p) h^{p/2-1} 	\E\left[   
		\int_0^t \xi^{(h)}(t,s)^2   \abs{a(X_s)}^p ds \right] \\
		{\revarn +}
		c(p)	h^{\alpha(p-1)} \E\left[ \int_0^t \abs{\xi^{(h)}(t,s)} \abs{b(X_s)}^p ds \right] 
	\end{multlined}
\end{align*}
by using \eqref{Eq : int g 1 h alpha} and \eqref{Eq : int g 2 h half} to get the second inequality.
Using again  \eqref{Eq : int g 1 h alpha}--\eqref{Eq : int g 2 h half}, we can write
 \begin{equation*}
 	{\revarn	\E\left[\abs*{Z^{(h)}_t-Z_t}^p\right]
 		\le c(p) {\revarn C^{p/2}_1} h^{p/2}  \sup_{s\in[0,T]} \E \left[|a(X_s)|^p \right]
 		 +
 		 c(p) {\revarn C^{p}_2} h^{\alpha p}  \sup_{s\in[0,T]} \E\left[ |b(X_s)|^p \right]},		
 \end{equation*}
 {\revarn where the constants $C_1, C_2$ above are the same as in Lemma 
 	\ref{lem : upper bound L1 L2 g minus 1} and depends on the kernel $K$ but not on the coefficients $a$ and $b$.}
 This is exactly \eqref{Eq : Lp Y - Z in statement}. 
 
 From the fact that $a$ and $b$ are at most with linear growth, using Assumption \ref{Ass : global Lip}, and recalling the control \eqref{Eq : moment Volterra} which is valid for $X$ we deduce
\begin{equation*} 
	\E\left[\abs*{Z^{(h)}_t-Z_t}^p\right]
	\le c(p)   h^{p/2} + c(p)  h^{p\alpha} \le c(p) h^{p/2},
\end{equation*}
{\revarn as $\alpha>1/2$, and for a constant $c(p)$ depending on the coefficients of $a$ and $b$, $K$.}  This proves \eqref{Eq : Lp Y - Z in statement particular} and the theorem follows. 
 \qed

\section{Statistical application}\label{S : Stat appl}
In this section we assume that $X^{\ve}$ is a solution to \eqref{Eq : Volterra SDE}. The process $X^\ve$ is observed on $[0,T]$ with a sampling step equal to $h>0$, under the assumption that $h\to0$. 

Let $X^{\ve,(h)}_t=X^\ve_{\varphi_h(t)}$ and denote  $Z^{\ve,(h)}_t=L\star(X^{{\revarn\ve}}_\cdot-x_0)(t)=\int_0^t L(t-s)(X^\ve_{\varphi_h(s)}-x_0)ds$. From the previous section, the process $Z^{\ve,(h)}$ is an approximation of the semimartingale
\begin{equation*}
	Z^\ve_t:=\ve \int_0^t a(X^\ve_s)dB_s+\int_0^t b(X^\ve_s,\theta^\star) ds.
\end{equation*}
We base our estimator of $\theta$ on the approximate observations 
of this semimartingale. The error of approximation of $Z^\ve$ by $Z^{\ve,(h)}$ is of magnitude $O(h^{1/2})$ by Theorem \ref{Th :  Lq norm Y^h Y}, which is the same magnitude as {\revarn an}  $h$-increment of the semimartingale $Z^\ve$.
Hence, we introduce a subsampling scheme. Let $k\ge 1$ and set $\Delta=kh$ with $k\to \infty$, $\Delta \to 0$. We denote $N=\lfloor T/\Delta \rfloor$.

Our estimation procedure is based on the random variables $Z^{\ve,(h)}_{j \Delta}$ and $X^{\ve}_{j\Delta}=X^{\ve}_{jkh}$ for $j=0,\dots,N$ which are available to the statistician.

We introduce the notation 
\begin{equation}\label{Eq : def Xi}
	\Xi_j^{\ve,(h)}(\theta)=
Z^{\ve,(h)}_{j\Delta + \Delta}-Z^{\ve,(h)}_{j\Delta}-
\Delta b(X^{\ve}_{j\Delta},\theta).
\end{equation}
We consider $H : \mathbb{R}^d \to \mathcal{S}^d$, where $\mathcal{S}^d$ is the linear space of symmetric $d\times d$ matrices and define
\begin{equation} \label{Eq : def contrast}
{\revarn	\mathbb{C}^{\ve,(h)}(\theta)=\sum_{j=0}^{N-1}
	\left(
	\Xi_j^{\ve,(h)}(\theta)^* H(X^{\ve}_{j\Delta})  \Xi_j^{\ve,(h)}(\theta)
	\right)}.
\end{equation}
The contrast function $\theta \mapsto 	\mathbb{C}^{\ve,(h)}(\theta)$ is real valued. We introduce $H$ to deal with a generic situation, however we mainly intend to consider $H(x)=Id$ {\modar or} $H(x)=(aa^*)^{-1}(x)$ when the latter is well defined.  
We introduce the following assumptions on $H$ :
\begin{assumption} \label{Ass : K mino}
The function $x \mapsto H(x)$ is $\mathcal{C}^1$
 and globally Lipschitz. Moreover,  we assume that $H$ is a {\revarn uniformly definite positive} symmetric matrix : $\exists c > 0, \forall x\in \mathbb{R}^d, H(x) \ge c Id$.
\end{assumption}
{\revarn Remark that from the globally Lipschitz property of $H$, there exists $C>0$ such that $\forall x\in \mathbb{R}^d, \abs{H(x)}\le C(1+\abs{x})$.}

We let $\widehat{\theta}_\ve= \argmin_{\theta \in \Theta} \mathbb{C}^{\ve,(h)}(\theta)$ be the estimator based on minimisation of the contrast.

We introduce the identifiability condition 
 \begin{assumption} \label{Ass : ident}
 	$\int_0^T \abs{b(X^0_s,\theta)-b(X^0_s,\theta^\star)}^2 ds > 0$ for all $\theta \neq \theta^\star$. 
 \end{assumption}
 {\revarn We precise that in the above assumption $X^0$ is the solution of \eqref{Eq : Volterra SDE} with $\ve=0$, and thus correspond to the value of the parameter $\theta^\star$.}

 We need some regularity on the coefficient of the Volterra SDE with respect to $(x,\theta)$.  
{\revarn For $E$ a finite dimensional space, and $k_1 \ge0, k_2\ge0$ integers, we denote by $\mathcal{C}^{k_1,k_2}_{\mathcal{P}}(\mathbb{R}^d\times\overset{\circ}{\Theta},E)$ the set of functions $f : \mathbb{R}^d\times\overset{\circ}{\Theta} \to E$ such that all
partial derivatives in $x$ of order up to $k_1$, and in $\theta$ 
of order up to $k_2$ exist, are continuous on $\mathbb{R}^d \times \overset{\circ}{\Theta}$, and have at most polynomial growth. 
More precisely, for any multi-indices $\bm{i}=(i_1,\dots,i_d) \in \mathbb{N}^d$ and  
$\bm{j}=(j_1,\dots,j_{d_\Theta}) \in \mathbb{N}^{d_\theta}$ with 
$\abs{\bm{i}}:=\sum_{l=1}^d i_l \le k_1$, and $\abs{\bm{j}}:=\sum_{l=1}^{d_\Theta} j_l \le k_2$, we assume
\begin{equation*}
	\sup_{\theta \in \overset{\circ}{\Theta}} 
	\norm*{
		\frac{\partial^{\abs*{\bm{i}}} f(x,\theta)}{\partial x^{i_1}_1 \dots \partial x^{i_d}_d } }_E +
	\sup_{\theta \in \overset{\circ}{\Theta}}	\norm*{
		\frac{\partial^{\abs*{\bm{j}}} f(x,\theta)}{\partial \theta^{j_1}_1 \dots\partial \theta^{j_{d_{\Theta}}}_{d_\Theta} }}_E \le c\left(1+\abs*{x}^c\right),
\end{equation*}
for some constant $c>0$. 
When mixed derivatives are needed, we denote for $k\ge 1$, by $\mathcal{C}^{k}_{\mathcal{P}}(\mathbb{R}^d\times\overset{\circ}{\Theta},E)$ the set of functions $f : \mathbb{R}^d\times\overset{\circ}{\Theta} \to E$ such that for all 
multi-indices $\bm{r}=(r_1,\dots,r_{d+d_\Theta})\in \mathbb{N}^{d+d_\Theta}$ with $\abs{\bm{r}}\le k$ 
the partial derivatives up to order $k$  exist and are continuous on $\mathbb{R}^d \times \overset{\circ}{\Theta}$. Moreover, we assume that there exists $c>0$ such that
$	\sup_{\theta \in \overset{\circ}{\Theta}} 
	\norm*{
		\frac{\partial^{\abs*{\bm{r}}} f(x,\theta)}{\partial x^{r_1}_1 \dots \partial x^{r_d}_d 
		\partial\theta^{r_{d+1}}_1 \dots\partial \theta^{r_{d+d_{\Theta}}}_{d_\Theta} }}
		_E \le c\left(1+\abs*{x}^c\right).$
}

 \subsection{Statistical results}

 We introduce the following assumption on the choice of the parameter $k$.
 \begin{assumption} 
 	\label{Ass: step consistency}
 	We assume $k=k_{{\revarn \ve}} \to \infty$ is a sequence of integers such that
 	$k=o(h^{-1})$,  $h^{\alpha-1}=o(k)$ and $\ve h^{-1/2} = o(k)$.
 \end{assumption}
  {\revarn Note that Assumption  \ref{Ass: step consistency} concerns the choice of the tuning parameter $k$ and not the statistical model itself.  The assumption 
 	is satisfied for instance if we choose $k\sim h^{-1/2}$, in which case $\Delta\sim h^{1/2}$.}

\begin{theorem} \label{Thm : consistency}
	Assume \ref{Ass : global Lip}, {\revK \ref{Ass : kernel resolvent},} \ref{Ass : K mino}, \ref{Ass : ident}, \ref{Ass: step consistency}, and {\revarn that} $b\in\mathcal{C}^{0,1}_{\mathcal{P}}(\mathbb{R}^d \times \overset{\circ}{\Theta})$.
Then, we have $\widehat{\theta}_\ve \xrightarrow[\mathbb{P}]{\ve\to {\revarn0}} \theta^\star$.		
\end{theorem}
{\revarn The conditions} to get a central limit theorem are more stringent. In particular, we need that the sampling step $h$ is small enough with regard to the noise size $\ve$. 
\begin{assumption} 
	\label{Ass: step normality}
	{\modar
		Assume that $h=o(\ve^\frac{1}{\alpha^2})$ and that $k=k_{{\revarn \ve}} \to \infty$ is a sequence of integers such that
	$k=o(\ve^{1/\alpha} h^{-1})$, $h^{\alpha-1} = o(k)$ and $\ve h^{-1/2}=o(k)$ .}	
\end{assumption}
{\revarn	Note that the condition $h=o(\ve^{\frac{1}{\alpha^2}})$ is implied by the first two conditions on the tuning parameter $k$ in Assumption \ref{Ass: step normality}. We nevertheless state it explicitly, since it corresponds to a fast sampling condition on the observation scheme of the model. 
	Conversely, the condition $h=o(\ve^{\frac{1}{\alpha^2}})$ guarantees the existence of a sequence $k$ satisfying the constraints given in Assumption \ref{Ass: step normality}. One may take $k=\ve^{1/\alpha}h^{-1} \gamma_{\ve,h}$ where $\gamma_{\ve,h}>0$ tends to zero sufficiently slowly. Also, note that Assumption \ref{Ass: step normality} implies Assumption  \ref{Ass: step consistency}.  
	}

For $(u,v)\in \{1,\dots,d_\Theta\}^2$, we define
\begin{equation*}
	I_{u,v}=\int_0^T \partial_{\theta_u} b(X^0_t,\theta^\star)^* H(X^0_t)\partial_{\theta_v} b(X^0_t,\theta^\star) dt.
\end{equation*}
We introduce the condition,
\begin{assumption} 
	\label{Ass: Fisher inver}
	The matrix $I=[I_{u,v}]_{1\le u , v \le d_\Theta}$ is invertible.
\end{assumption}
We define ${\revarn \mathcal{G}:=} \int_0^T  \partial_\theta b(X^0_s,\theta^\star)^*
H(X^0_s)a(X_s^0)d B_s$ which is a {\revarn ${d_\Theta}$-dimensional Gaussian random variable}. 
\begin{theorem} \label{Thm : normality}
	Assume \ref{Ass : global Lip},  {\revK \ref{Ass : kernel resolvent},} \ref{Ass : K mino}, \ref{Ass : ident}, \ref{Ass: step normality},
	\ref{Ass: Fisher inver},  and 
	 that $b\in\mathcal{C}^{2}_{\mathcal{P}}(\mathbb{R}^d \times \overset{\circ}{\Theta}) \cap \mathcal{C}^{0,3}_{\mathcal{P}}(\mathbb{R}^d \times \overset{\circ}{\Theta})$. 
	Then, we have 
	\begin{equation*} 
	\ve^{-1}\left(\widehat{\theta}_\ve - \theta^\star\right) \xrightarrow[\mathbb{P}]{\ve\to 0} 
	 I^{-1} {\revarn \mathcal{G}.}
	\end{equation*}
\end{theorem}
\begin{remark}
	{\revarn The asymptotic behavior of the estimator is the same as if
		the semimartingale $Z^{\ve}$ were exactly observed.}
	The rate of estimation is {\revarn $\ve$, independently of $\alpha$,}	 which is the rate classically obtained in the situation of drift estimation for semimartingale 
	under small noise asymptotic {\revarn (see \cite{kutoyantsIdentificationDynamicalSystems1994})}. 
	Some {\revarn conditions are} required on the sampling step $h$, which is {\revarn consistent with earlier results for}
	diffusion processes {\revarn (see \cite{sorensenSmalldiffusionAsymptoticsDiscretely2003b})}. Our condition means that the sampling step $h$ must be small {\revarn relative} to the $\ve$ which {\revarn determines} the rate of the estimator. Our condition is more restrictive when $\alpha$ is small and reads {\modar $h=o(\ve^4)$} for $\alpha \simeq 1/2$ and {\modar $h=o(\ve)$} for $\alpha \simeq 1$.  
	The condition on $h$ arises, due to the error in 
the reconstruction of the semimartingale $Z^\ve$ from the observed sampling of $X^\ve$, {\revarn as well as from discretization error of the trajectory $X^\ve$}. Let us stress that in Remark
	 \ref{rem : certainly rate optimal}, we conjecture that our upper bound on the reconstruction error of {\revarn $Z^\ve$} is {\revarn sharp.}
\end{remark}

\begin{remark}
	\label{R : Fisher} 
	{\revarn Let  $\bar{I}$ be the covariance matrix of $\mathcal{G}$ which is given by
		$\bar{I}=\int_0^T \partial_\theta b(X^0_s,\theta^\star)^*
		H(X^0_s)(aa^*)(X_s^0)H(X^0_s)\partial_\theta b(X^0_s,\theta^\star)ds$. The covariance matrix of $I^{-1}\mathcal{G}$ is therefore $I^{-1}\bar{I}I^{-1}$.}	
	If one takes $H(x)=(aa^*)^{-1}(x)$, in the situation where the latter is well defined, we have
	$I=\bar{I}=I^{F}$ where $I^{F}$ is the Fisher information associated with the continuous observation of $Z^\ve$,
	 \begin{equation*}
		I^{F}:= \left[	
\int_0^T \partial_{\theta_u} b(X^0_t,\theta^\star)^* (aa^*)(X^0_t)^{-1}\partial_{\theta_v} b(X^0_t,\theta^\star) dt\right]_{1\le u,v \le d_\theta}
\end{equation*}
In this case, we have 
 $\ve^{-1}\left(\widehat{\theta}_\ve - \theta^\star\right) \xrightarrow[\mathcal{L}]{\ve \to0}
 \mathcal{N}(0,(I^F)^{-1})$. 
 
 {\revarn Note that the continuous observations of the processes $Z^\ve$ and $X^\ve$ are statistically equivalent. Indeed, 
 	a path of one process can be transformed into a path of the other by convolution with the kernels $K$ or $L$, which are independent of the parameter $\theta$. 
 	The optimality of the inverse Fisher information matrix 
 	is established, in the context of the observing a semimartingale process, in Section 2.1 of \cite{kutoyantsIdentificationDynamicalSystems1994}.  	
 	We deduce that $(I^F)^{-1}$ is the minimal variance of estimation, based on the continuous observation of $Z^\ve$, and thus based on $X^\ve$ as well.
 	Since the estimator $\widehat{\theta}_\ve$ reaches this minimal variance from a discrete sampling of $X^\ve$, it is efficient. 	
 	
 	Remark also that the variance of the estimator depends implicitly on $\alpha$ through the path $X^0$. 
 } 
\end{remark}
{\revarn
\begin{remark}
	The statistical method relies on knowledge of the kernel. Thus, the estimation of the kernel, or of the roughness level $\alpha$, can be an issue. For results related to this statistical problem, we refer to \cite{chongStatisticalInferenceRough2024, corcueraAsymptoticTheoryBrownian2013}.
\end{remark}
}

\subsection{Proofs of the statistical results}
Before proving Theorem{\revarn s} \ref{Thm : consistency} and \ref{Thm : normality}, we have to
introduce some notations.

First we let for all $t\in [0,T]$, 
\begin{equation}
	\label{Eq : def Rt ve h}
R_{t}^{\ve,(h)}=Z^{\ve,(h)}_t-Z^\ve_t.
\end{equation}
 The following lemma is obtained exactly as {\revarn we deduce \eqref{Eq : Lp Y - Z in statement particular} from \eqref{Eq : Lp Y - Z in statement} in
 Theorem \ref{Th :  Lq norm Y^h Y},}
 accounting from the fact
that the diffusion coefficient $a$ is replaced by $\ve a$ in \eqref{Eq : Lp Y - Z in statement},
\begin{lemma}\label{Lem : upper bound Rve}
	 Assume {\revK \ref{Ass : global Lip} and \ref{Ass : kernel resolvent}}. For all $p\ge 1$, there exists $c(p) \ge 0$ such that for all $0<\ve\le 1$, 
	 $0<{\revarn t}\le T$.
\begin{equation*}
		\E\left[ \abs*{R_{t}^{\ve,(h)}}^p \right] \le c(p) ( \ve^p h^{p/2} + h^{\alpha p}  ). 
	\end{equation*}
\end{lemma}
We also need two lemma{\revarn s} on {\revarn the} increments of $X^\ve$ and {\revarn on the} rate of approximation of $X^0$ by $X^\ve$.
\begin{lemma} \label{lem : shorter lemma increments}
		Assume \ref{Ass : global Lip}, then for all $p\ge 1$, there exists $c(p) \ge 0$ 
such that for all $0\le s \le t \le T$,
\begin{equation} \label{eq : control incre Xve}
	\E\left[\abs{X_{t}^\ve-{X_{s}^\ve}}^p\right] \le c(p) [\ve^p (t-s)^{(\alpha-1/2)p} + (t-s)^{\alpha p}].
\end{equation}
\end{lemma}
\begin{proof} It is a straightforward consequence of Lemma \ref{Lem : cond moment process X} in the Appendix.
\end{proof}
We recall a result on the convergence of $X^\ve$ as $\ve \to 0$.
\begin{lemma}[Lemma 3.2 in \cite{gloterDriftEstimationRough2025}]
	\label{Lem : conv Xve X0}
	Assume \ref{Ass : global Lip}, then for all $p\ge 1$, there exists $c(p) \ge 0$ such that
	\begin{equation*}
		\sup_{t\in[0,T]} \E\left[\abs*{X^\ve_t-X^0_t}^p\right] \le c(p) \ve^p.
	\end{equation*}
\end{lemma}
\subsubsection{Convergence of the contrast function}
\begin{lemma}\label{Lem :  convergence contrast}
	Assume \ref{Ass : global Lip}, {\revK \ref{Ass : kernel resolvent},} \ref{Ass : K mino} and \ref{Ass: step consistency}.
	Then, for all $\theta\in\Theta$, we have
	\begin{multline} \label{Eq : convergence contrast}
		\Delta^{-1}(	\mathbb{C}^{\ve,(h)}(\theta)-	\mathbb{C}^{\ve,(h)}(\theta^\star))
		\\	\xrightarrow[\P]{\ve \to 0} \int_0^T[b(X^0_s,\theta)-b(X^0_s,\theta^\star)]^*H(X^0_s)
		[b(X^0_s,\theta)-b(X^0_s,\theta^\star)] {\revarn ds}.
	\end{multline}
\end{lemma}
\begin{proof}
From the definitions \eqref{Eq : def Xi} and \eqref{Eq : def contrast}, we have
\begin{multline*}
	\mathbb{C}^{\ve,(h)}(\theta)=	\mathbb{C}^{\ve,(h)}(\theta^\star)
	+ 2 \Delta \sum_{j=0}^{N-1} \Xi_j^{\ve,(h)}(\theta^\star)^* H(X^\ve_{j\Delta})(b(X^\ve_{j\Delta},\theta) -b(X^\ve_{j\Delta},\theta^\star))
	\\
	{\revarn +}
	\Delta^2 \sum_{j=0}^{N-1} (b(X^\ve_{j\Delta},\theta) -b(X^\ve_{j\Delta},\theta^\star))^* H(X^\ve_{j\Delta})(b(X^\ve_{j\Delta},\theta) -b(X^\ve_{j\Delta},\theta^\star)).
\end{multline*}
From this we write $	\Delta^{-1}(	\mathbb{C}^{\ve,(h)}(\theta)-	\mathbb{C}^{\ve,(h)}(\theta^\star))= \mathcal{I}^{\ve,1}_N(\theta)+\mathcal{I}^{\ve,2}_N(\theta)$,
where
\begin{align*}
	\mathcal{I}^{\ve,1}_N(\theta)=&2 \sum_{j=0}^{N-1} \Xi_j^{\ve,(h)}(\theta^\star)^*
	H(X^\ve_{j\Delta})(b(X^\ve_{j\Delta},\theta) -b(X^\ve_{j\Delta},\theta^\star)),
	\\
	\mathcal{I}^{\ve,2}_N(\theta)=&\Delta\sum_{j=0}^{N-1} (b(X^\ve_{j\Delta},\theta) -b(X^\ve_{j\Delta},\theta^\star))^* H(X^\ve_{j\Delta})(b(X^\ve_{j\Delta},\theta) -b(X^\ve_{j\Delta},\theta^\star)).
\end{align*}
We study the convergence of $	\mathcal{I}^{\ve,2}_N(\theta)$ which is the main term. 
We use that $b(\cdot,\theta)$, $b(\cdot,\theta^\star)$ and $H(\cdot)$ are Lipschitz functions and that $\sup_{\theta \in \Theta} \abs{b(x,\theta)}+\abs{H(x)}\le C(1+\abs{x})$ from the Assumptions \ref{Ass : global Lip} {\revK and \ref{Ass : K mino}}. It enables us to get 
\begin{equation*}
	\abs*{\mathcal{I}^{\ve,2}_N(\theta)-\mathcal{I}^{0,2}_N(\theta)}
	 \\ \le C \Delta\sum_{j=0}^{N-1} \abs*{X^\ve_{j\Delta}-X^0_{j\Delta}} \left(1+\abs*{X^\ve_{j\Delta}}^2+\abs*{X^0_{j\Delta}}^2\right) 
\end{equation*}
Taking the expectation and using \eqref{Eq : moment Volterra} with Lemma \ref{Lem : conv Xve X0}, we deduce
\begin{align*}
\E \left[\abs*{\mathcal{I}^{\ve,2}_N(\theta)-\mathcal{I}^{0,2}_N(\theta)}\right]
&\le \Delta \sum_{j=0}^{N-1} C \E\left[\abs*{X^\ve_{j\Delta}-X^0_{j\Delta}}^2\right]^{1/2}
 \E\left[(1+\abs*{X^\ve_{j\Delta}}^4+\abs*{X^0_{j\Delta}}^4) \right]^{1/2}
 \\
 &\le C \Delta \sum_{j=0}^{N-1} \ve \le C \ve,
\end{align*}
as $\Delta N =O(1)$. 
It remains to study the convergence of the deterministic quantity $\mathcal{I}^{0,2}_N(\theta)$. We recognize a Riemann sum and using that $t \mapsto X^0_s$ is continuous, we deduce that $\mathcal{I}^{0,2}_N(\theta) \xrightarrow{N \to \infty}
  \int_0^T[b(X^0_s,\theta)-b(X^0_s,\theta^\star)]^*H(X^0_s)
[b(X^0_s,\theta)-b(X^0_s,\theta^\star)]ds$.
Consequently $\mathcal{I}^{\ve,2}_N(\theta)$ is converging in probability to the RHS of 
\eqref{Eq : convergence contrast}.

To end the proof of the proposition, we have to show the convergence to $0$ of 	$\mathcal{I}^{\ve,1}_N(\theta)$.
From \eqref{Eq : def Xi} and \eqref{Eq : def Rt ve h}, we have ${\revarn \Xi^{\ve,(h)}_{j}(\theta)}=Z^{\ve,(h)}_{j\Delta + \Delta}-Z^{\ve,(h)}_{j\Delta}-
\Delta b(X^{\ve}_{j\Delta},\theta)=
 Z^{{\revarn\ve}}_{j\Delta + \Delta}-Z^{{\revarn\ve}}_{j\Delta}-
\Delta b(X^{\ve}_{j\Delta},\theta) + R^{\ve,(h)}_{(j+1)\Delta}-R^{\ve,(h)}_{j\Delta}$. Inserting the dynamic of $Z^{{\revarn \ve}}$ into this equation yields 
\begin{multline}\label{Eq : Xi expanded}
	{\revarn \Xi^{\ve,(h)}_{j}(\theta)}=\ve
	\int_{j\Delta}^{(j+1)\Delta} a(X^\ve_s) dB_s + \int_{j\Delta}^{(j+1)\Delta} (b(X_s^\ve,\theta^\star) - b(X_{j\Delta}^\ve,\theta)) ds\\
	+ R^{\ve,(h)}_{(j+1)\Delta}-R^{\ve,(h)}_{j\Delta}.
\end{multline}
Let us denote $\xi_{{\revarn j}}:={\revarn \Xi^{\ve,(h)}_{j}(\theta^\star)^*} 	H(X^\ve_{j\Delta})(b(X^\ve_{j\Delta},\theta) -b(X^\ve_{j\Delta},\theta^\star))$, that we split into $\xi_{{\revarn j}}=\sum_{l=1}^3\overline{\xi}_{{\revarn j}}^l$ where
\begin{align*}
	\overline{\xi}_{{\revarn j}}^1&=\ve\left( \int_{j\Delta}^{(j+1)\Delta} a(X^\ve_s) dB_s\right)^{\revarn *} H(X^\ve_{j\Delta})(b(X^\ve_{j\Delta},\theta) -b(X^\ve_{j\Delta},\theta^\star)),
	\\
		\overline{\xi}_{{\revarn j}}^2&=
		\left( \int_{j\Delta}^{(j+1)\Delta} (b(X_s^\ve,\theta^{{\revarn\star}}) - b(X_{j\Delta}^\ve,\theta^{{\revarn\star}})) ds \right)^{\revarn *} H(X^\ve_{j\Delta})(b(X^\ve_{j\Delta},\theta) -b(X^\ve_{j\Delta},\theta^\star)),
		\\
			\overline{\xi}_{{\revarn j}}^3&=\big(R^{\ve,(h)}_{(j+1)\Delta}-R^{\ve,(h)}_{j\Delta}\big)^{\revarn *}
			H(X^\ve_{j\Delta})(b(X^\ve_{j\Delta},\theta) -b(X^\ve_{j\Delta},\theta^\star)).
\end{align*}
Let us focus on $\sum_{j=0}^{N-1} \overline{\xi}_{{\revarn j}}^1$.
As  $\overline{\xi}_{{\revarn j}}^1 $ is a
 $\mathcal{F}_{(j+1)\Delta}$-measurable variable, by Lemma 9 in 
 \cite{genon-catalotEstimationDiffusionCoefficient1993}, it is sufficient to prove
\begin{align} \label{Eq : cond xi moment 1}
	\sum_{j=0}^{N-1} \abs*{\E\left[\overline{\xi}_{{\revarn j}}^1 \mid \mathcal{F}_{j\Delta} \right]}\xrightarrow{\mathbb{P}} 0,
	\\ \label{Eq : cond xi moment 2}
	\sum_{j=0}^{N-1} \E\left[\abs[\big]{\overline{\xi}_{{\revarn j}}^1}^2  \mid \mathcal{F}_{j\Delta} \right]\xrightarrow{\mathbb{P}} 0.
\end{align}
We have $\E\left[\overline{\xi}_{{\revarn j}}^1 \mid \mathcal{F}_{j\Delta} \right]=0$ and \eqref{Eq : cond xi moment 1} is immediate. Moreover,
\begin{multline*} \E\left[\abs[\big]{\overline{\xi}_{{\revarn j}}^1}^2  \mid \mathcal{F}_{j\Delta} \right]
\le c \ve^2\E\left[\int_{j\Delta}^{(j+1)\Delta} \abs*{a(X^\ve_s)}^2 ds \mid \mathcal{F}_{j\Delta} \right]  \\
{\revarn \times}
\abs*{H(X^\ve_{j\Delta})}^2 \abs*{b(X^\ve_{j\Delta},\theta)
	 -b(X^\ve_{j\Delta},\theta^\star)}^2.
\end{multline*}
We deduce that 
\begin{multline*}\E\left[ \E\left[\abs[\big]{\overline{\xi}_{{\revarn j}}^1}^2  \mid \mathcal{F}_{j\Delta} \right]\right] \\ \le c \ve^2
	 \int_{j\Delta}^{(j+1)\Delta} \E\left[ \abs*{a(X^\ve_s)}^2 \abs*{H(X^\ve_{j\Delta})}^2 \abs*{b(X^\ve_{j\Delta},\theta)
	-b(X^\ve_{j\Delta},\theta^\star)}^2 \right] ds \le c \ve^2 \Delta  
\end{multline*}
using that $b$ and $H$ have at most linear growth and \eqref{Eq : moment Volterra}.
Then,
\begin{equation*}
\E \left[ \abs*{\sum_{j=0}^{N-1} \E\left[\abs[\big]{\overline{\xi}_{{\revarn j}}^1}^2  \mid \mathcal{F}_{j\Delta} \right]} \right] \le c \sum_{j=0}^{N-1} \ve^2 \Delta \le {\revarn c} \ve^2.
\end{equation*}
Consequently the convergence \eqref{Eq : cond xi moment 2} holds in ${\revarn\mathbf{L}}^1(\Omega)$ and thus in probability. It yields
$\sum_{j=0}^{N-1} \overline{\xi}_{{\revarn j}}^1 \xrightarrow{\mathbb{P}}0$.

Now, we study $\sum_{j=0}^{N-1} \overline{\xi}_{{\revarn j}}^2$.  By using Lemma 
\ref{lem : shorter lemma increments} and {\revarn the} Lipschitz property of $b$, we write
\begin{multline*}
	\E\left[\abs*{\int_{j\Delta}^{(j+1)\Delta}(b(X_s^\ve,\theta^\star) - b(X_{j\Delta}^\ve,\theta^\star))ds }^{\modar 4}    \right]  
	\le 	c\Delta^{\modar 3}\int_{j\Delta}^{(j+1)\Delta}
	\E\left[\abs*{X^\ve_{s}-X_{j\Delta}^\ve}^{\modar 4} \right]ds \\
	\le c\Delta^{\modar 4} [  \ve^{\modar 4} \Delta^{{\modar 4}(\alpha-1/2)}+\Delta^{{\modar 4}\alpha}] =
	c [  \ve^{\modar 4} \Delta^{{\modar 4\alpha+2}}+ \Delta^{{\modar 4}\alpha+{\modar 4}}] 
\end{multline*}
Using that $b$ and $H$ have at most linear growth, \eqref{Eq : moment Volterra} and Cauchy--Schwarz's inequality, we deduce that
\begin{equation*}
	\E\left[\abs[\big]{\overline{\xi}_{{\revarn j}}^2} \right] \le c (\ve^{\modar 2} \Delta^{{\modar 2\alpha+1}}+\Delta^{{\modar 2\alpha+2}}).
\end{equation*}
{\revarn Together with $N\Delta=O(1)$,} it entails $\sum_{j=0}^{N-1}\overline{\xi}_{{\revarn j}}^2\xrightarrow{{\revarn\mathbf{L}}^1(\Omega)}0$.

It remains to prove the convergence to zero of $\sum_{j=0}^{N-1} \overline{\xi}_{{\revarn j}}^3$.
Using H\"older inequality, \eqref{Eq : moment Volterra} and Lemma \ref{Lem : upper bound Rve},
 we have $\E\left[\abs[\big]{\overline{\xi}_{{\revarn j}}^3} \right] \le c[\ve h^{1/2}+h^{\alpha}]$.
We deduce $\norm*{\sum_{j=0}^{N-1} \overline{\xi}_{{\revarn j}}^3}_{{\revarn\mathbf{L}}^1(\Omega)} \le c [\ve N h^{1/2} + Nh^\alpha ]$. As $N=T/\Delta=T/(kh)$ it implies
$\norm{\sum_{j=0}^{N-1} \overline{\xi}_{{\revarn j}}^3}_{{\revarn\mathbf{L}}^1(\Omega)} \le c [\ve h^{-1/2}k^{-1} + h^{\alpha-1}k^{-1}]$. This quantity tends to $0$ 
by Assumption \ref{Ass: step consistency}. The proposition is proved.
\end{proof}
\begin{lemma}\label{lem : Lp bound der contrast}
		Assume \ref{Ass : global Lip}, {\revK \ref{Ass : kernel resolvent},} \ref{Ass : K mino} and \ref{Ass: step consistency} and that $b\in\mathcal{C}^{0,1}_\mathcal{P}(\mathbb{R}^{{\revarn d}}\times\overset{\circ}{\Theta})$.
	Then, for all $p \ge 1$, there exists $c(p)$ such that, we have
	\begin{equation*} 
		\E\left[ \abs*{\frac{\mathbb{C}^{\ve,(h)}(\theta)-\mathbb{C}^{\ve,(h)}(\theta')}{\Delta}}^p\right]
		\le c(p) \abs*{\theta-\theta'}^p, \quad \forall \theta, \theta' \in \Theta.
	\end{equation*}	
\end{lemma}	
\begin{proof}
From \eqref{Eq : def contrast}, if we set $b(X^\ve_{j\Delta},\theta,\theta')= b(X_{j\Delta}^\ve,\theta)-b(X_{j\Delta}^\ve,\theta')$, we have
\begin{align*}
 \frac{1}{\Delta}\left[\mathbb{C}^{\ve,(h)}(\theta) - \mathbb{C}^{\ve,(h)}(\theta')\right]&=
 \begin{multlined}[t]	
2 \sum_{j=0}^{N-1}
b(X^\ve_{j\Delta},\theta,\theta')^* H(X^{\ve}_{j\Delta})  \Xi_j^{\ve,(h)}(\theta)
\\ +
\Delta
\sum_{j=0}^{N-1}
b(X^\ve_{j\Delta},\theta,\theta')^* H(X^{\ve}_{j\Delta}) {\revarn b(X^\ve_{j\Delta},\theta,\theta')}
\end{multlined}\\
&=\sum_{l=1}^4 \mathbb{D}^{\ve,(h),l}(\theta,\theta'),
\end{align*}
where using \eqref{Eq : Xi expanded}, we have set    
\begin{align*}
	\mathbb{D}^{\ve,(h),1}(\theta,\theta') &= 2 \ve \sum_{j=0}^{N-1}
	b(X^\ve_{j\Delta},\theta,\theta')^* H(X^{\ve}_{j\Delta}) \int_{j\Delta}^{(j+1)\Delta} a(X^\ve_s) dB_s ,
	\\
	\mathbb{D}^{\ve,(h),2}(\theta,\theta') &=2 \sum_{j=0}^{N-1}
	b(X^\ve_{j\Delta},\theta,\theta')^* H(X^{\ve}_{j\Delta}) \int_{j\Delta}^{(j+1)\Delta} [b(X^\ve_s,\theta^\star)-b(X^\ve_{j\Delta},\theta)] ds,
	\\
	\mathbb{D}^{\ve,(h),3}(\theta,\theta') &=2 \sum_{j=0}^{N-1}
	b(X^\ve_{j\Delta},\theta,\theta')^* H(X^{\ve}_{j\Delta}) [R^{\ve,(h)}_{(j+1)\Delta}-R^{\ve,(h)}_{j\Delta}],
	\\
	\mathbb{D}^{\ve,(h),4}(\theta,\theta') &=\Delta
	\sum_{j=0}^{N-1}
	b(X^\ve_{j\Delta},\theta,\theta')^* H(X^{\ve}_{j\Delta}) {\revarn b(X^\ve_{j\Delta},\theta,\theta')}.
\end{align*}
	From Burkholder-Davis-Gundy's inequality, we have
\begin{multline*}
	\E\left[\abs*{	\mathbb{D}^{\ve,(h),1}(\theta,\theta')}^p\right] 
 \\ \le c(p) \ve^p \E \left[ \sum_{j=0}^{N-1} \int_{j\Delta}^{(j+1)\Delta} \abs*{b(X^\ve_{j\Delta},\theta,\theta')}^2 \abs*{H(X^{\ve}_{j\Delta})}^2 \abs{a(X^\ve_s)}^2 ds \right]^{p/2}
\end{multline*}
Using Jensen's inequality with $N\Delta \le T$ implies,
it is smaller than\\
$  c(p) \ve^p
\sum_{j=0}^{N-1} 
\int_{j\Delta}^{(j+1)\Delta} \E \left[\abs*{b(X^\ve_{j\Delta},\theta,\theta')}^p \abs*{H(X^{\ve}_{j\Delta})}^p \abs*{a(X^\ve_s)}^p \right] ds$. 
As $b \in \mathcal{C}^{0,1}_{\mathcal{P}}$ together with $|H|^p$ and $|a|^p$ {\revarn having} sub-polynomial growth, the expectation in the previous equation is lower than
${\revarn c}\abs*{\theta-\theta'}^p 
\int_{j\Delta}^{(j+1)\Delta} \E \left[ (1+|X_{j\Delta}^\ve|^c)(1+|X_{s}^\ve|^c)\right]ds$ for some $c=c(p)>0$.
By \eqref{Eq : moment Volterra}, we deduce 
	$\E\left[\abs*{	\mathbb{D}^{\ve,(h),1}(\theta,\theta')}^p\right] 
	\le c(p) \ve^p|\theta-\theta'|^p$.

{\revarn Using the triangle} inequality, followed by the Cauchy--Schwarz inequality, we have
$\norm{	\mathbb{D}^{\ve,(h),2}(\theta,\theta')}_{{\revarn\mathbf{L}}^p(\Omega)}
\le c(p) \Delta \sum_{j=0}^{N-1} 
\norm{b(X^\ve_{j\Delta},\theta,\theta') }_{{\revarn\mathbf{L}}^{2p}(\Omega)}$
where we used that the ${\revarn\mathbf{L}}^{2p}(\Omega)$ norm of $H(X^\ve_{j\Delta}) 
	\int_{j\Delta}^{(j+1)\Delta} [b(X^\ve_s,\theta^\star)-b(X^\ve_{j\Delta},\theta)]ds$ is bounded by $c(p) \Delta$ for some constant $c(p)$. As $b \in \mathcal{C}^{0,1}_{\mathcal{P}}$, and $N\Delta \le T$, we deduce that
	$\norm{	\mathbb{D}^{\ve,(h),2}(\theta,\theta')}_{{\revarn\mathbf{L}}^p(\Omega)} \le c(p) \abs*{\theta-\theta'}$. With analogous computations, we show $\norm{	\mathbb{D}^{\ve,(h),4}(\theta,\theta')}_{{\revarn\mathbf{L}}^p(\Omega)} \le c(p) \abs*{\theta-\theta'}$.
For the last term, we write using that $b\in\mathcal{C}^{0,1}_{\mathcal{P}}$ and $H$ are at most with linear growth,
\begin{align*}
	\norm{	\mathbb{D}^{\ve,(h),3}(\theta,\theta') }_{{\revarn\mathbf{L}}^p(\Omega)} &\le c(p)
\abs{\theta-\theta'}	\sum_{j=0}^{N-1} \big[\norm{R^{\ve,(h)}_{(j+1)\Delta}}_{{\revarn\mathbf{L}}^{2p}(\Omega)} +\norm{R^{\ve,(h)}_{j\Delta}}_{{\revarn\mathbf{L}}^{2p}(\Omega)}\big]
	\\
	 &\le c(p)\abs{\theta-\theta'}		\sum_{j=0}^{N-1} [\ve h^{1/2}+h^\alpha], \text { by Lemma \ref{Lem : upper bound Rve},}
	 \\
	 &\le c(p)\abs{\theta-\theta'} 	N [\ve h^{1/2}+h^\alpha] \\ &\le c(p)\abs{\theta-\theta'}  [\ve k^{-1}h^{-1/2}+k^{-1}h^{\alpha-1}],
\end{align*}	
where in the last line we used  $N=O(1/\Delta)=O(k^{-1}h^{-1})$. From Assumption \ref{Ass: step consistency}, we deduce $	\norm{	\mathbb{D}^{\ve,(h),3}(\theta,\theta') }_{{\revarn\mathbf{L}}^p(\Omega)}\le c(p)\abs{\theta-\theta'}$. This concludes the proof of the lemma.
%
%
\end{proof}

\subsubsection{Proof of Theorem \ref{Thm : consistency}}
{\revarn In this proof, we set, for $\ve>0$, $\mathbb{C}^{\ve}(\theta,\theta^\star):=	\Delta^{-1}(\mathbb{C}^{\ve,(h)}(\theta)-\mathbb{C}^{\ve,(h)}(\theta^\star))$ and 
denote by $\mathbb{C}^{0}(\theta,\theta^\star)$} the right-hand side of \eqref{Eq : convergence contrast}. From Assumption \ref{Ass : K mino}, we have $\mathbb{C}^{0}(\theta,\theta^\star) \ge c \int_0^T \abs*{b(X^0_s,\theta)-b(X^0_s,\theta^\star)}^2ds$, which admits a unique minimum at $\theta=\theta^\star$ under \ref{Ass : ident}.
{\revarn From Lemma \ref{Lem :  convergence contrast}, the contrast function $\mathbb{C}^{\ve}(\theta,\theta^\star)$ converges in 
	$\mathbf{L}^{1}(\Omega)$ to $\mathbb{C}^{0}(\theta,\theta^\star)$ for each $\theta \in \Theta$, as $\ve\to0$. 
	Let us consider the continuity modulus $w_{\eta}^{\ve}:=\sup_{\abs*{\theta-\theta'}\le \eta}
	\abs*{\mathbb{C}^{\ve}(\theta,\theta^\star)-\mathbb{C}^{\ve}(\theta',\theta^\star)}$, for $\eta>0$.
Lemma \ref{lem : Lp bound der contrast} enables us to apply Theorem 19 in Appendix I of \cite{IbragimovHasminkiiBook81}. We deduce that
$\sup_{\ve \in(0,1]} \E\left[  w_{\eta}^{\ve}  \right] \xrightarrow{\eta\to0}0$.}
It ensures that the convergence {\revarn of the contrast function  $\mathbb{C}^{\ve}(\cdot,\theta^\star)$} is uniform on the compact convex set $\Theta$. 
 The consistency of the minimum contrast estimator {\revarn then} follows from the uniform convergence of the contrast function.
\qed

\subsubsection{Asymptotic normality of the quasi-score function}

\begin{lemma} \label{lem : conv quasi score} 
	Assume \ref{Ass : global Lip}, {\revK \ref{Ass : kernel resolvent},} \ref{Ass : K mino}, \ref{Ass: step normality} and that 
	{\modar $b \in \mathcal{C}^{0,1}_\mathcal{P}\big(\mathbb{R}^d \times \overset{\circ}{\Theta}\big) \cap \mathcal{C}^{2}_{\mathcal{P}}(\mathbb{R}^d\times\overset{\circ}{\Theta})$. } Then, we have
	\begin{equation} \label{eq : conv quasi score}
		\frac{1}{\ve \Delta} \partial_\theta \mathbb{C}^{\ve,(h)}(\theta^\star)
		\xrightarrow{\mathbb{P}} -2 \int_0^T  \partial_\theta b(X^0_s,\theta^\star)^*
		H(X^0_s)a(X_s^0)d B_s.
	\end{equation}
\end{lemma}
\begin{proof}
	From \eqref{Eq : def Xi}--\eqref{Eq : def contrast}, we have
	\begin{equation}\label{Eq : partial C}
		 \frac{1}{\ve \Delta} \partial_\theta \mathbb{C}^{\ve,(h)}(\theta)= -\frac{2}{\ve} \sum_{j=0}^{N-1} {\revarn \partial_\theta b(X^\ve_{j\Delta},\theta)^*   H (X^\ve_{j\Delta}) \Xi_{j}^{\ve,(h)}(\theta).}  
	\end{equation}
	Using \eqref{Eq : Xi expanded},	
	 we deduce $ \frac{1}{\ve \Delta} \partial_\theta \mathbb{C}^{\ve,(h)}(\theta^\star)=\sum_{l=1}^3 	\mathbb{F}^{\ve,(h),l}(\theta^\star)$ with
	\begin{align} \nonumber
		\mathbb{F}^{\ve,(h),1}(\theta^\star) &= -2 \sum_{j=0}^{N-1}
		\partial_\theta b(X^\ve_{j\Delta},\theta^\star)^* H(X^{\ve}_{j\Delta}) \int_{j\Delta}^{(j+1)\Delta} a(X^\ve_s) dB_s ,
		\\ \nonumber
		\mathbb{F}^{\ve,(h),2}(\theta^\star) &=-\frac{2}{\ve} \sum_{j=0}^{N-1}
		\partial_\theta b(X^\ve_{j\Delta},\theta^\star)^* H(X^{\ve}_{j\Delta}) \int_{j\Delta}^{(j+1)\Delta} \big[ b(X^\ve_s,\theta^\star) -b(X^\ve_{j\Delta},\theta^\star) \big] ds,
		\\ \label{eq : def F 3 in score}
		\mathbb{F}^{\ve,(h),3}(\theta^\star) &=-\frac{2}{\ve} \sum_{j=0}^{N-1}
		\partial_\theta b(X^\ve_{j\Delta},\theta^\star)^* H(X^{\ve}_{j\Delta}) [R^{\ve,(h)}_{(j+1)\Delta}-R^{\ve,(h)}_{j\Delta}].
	\end{align}
	From $b\in \mathcal{C}^{0,1}_\mathcal{P}$, Assumption \ref{Ass : K mino} and Lemma \ref{Lem : conv Xve X0}, we deduce the convergence 
	{\revarn in probability} of
	$\mathbb{F}^{\ve,(h),1}(\theta^\star)$ to the {\revarn right hand-side} of \eqref{eq : conv quasi score}.
	
	We now prove the convergence to $0$ in ${\revarn\mathbf{L}}^1(\Omega)$ for the two other terms.	
	 Using Assumption \ref{Ass : global Lip}, and polynomial growth of 
	$\partial_\theta b$ and $H$, we get
	\begin{equation*}
		\E\left[\abs*{\mathbb{F}^{\ve,(h),2}(\theta^\star)}\right]
\le 
\frac{C}{\ve} \sum_{j=0}^{N-1} \\ 
\int_{j\Delta}^{(j+1)\Delta} 
\E\left[\left(1+\abs*{X^\ve_{j\Delta}}^C\right) \abs*{X^\ve_{s}-X^\ve_{j\Delta}}\right]ds		
	\end{equation*}
	By \eqref{eq : control incre Xve}, we deduce
$	\E\left[\abs[\big]{\mathbb{F}^{\ve,(h),2}(\theta^\star)}\right] \le C \ve^{-1} N \Delta (\ve \Delta^{\alpha-1/2}+\Delta^\alpha) \le C [\Delta^{\alpha-1/2} + \ve^{-1}\Delta^\alpha]$. This quantity goes to zero as $\Delta \to 0$ and $\ve^{-1}\Delta^\alpha=O(\ve^{-1}k^\alpha h^\alpha) =o(1)$ by Assumption \ref{Ass: step normality}.

{\modar 	
For the last term, we set $J(x)=\partial_\theta b(x,\theta^\star)^*H(x)$. Since $b \in \mathcal{C}^2_\mathcal{P}(\mathbb{R}^d \times \overset{\circ}{\Theta})$ and recalling Assumption 
\ref{Ass : K mino}, we have $\abs*{J(x)-J(x')}\le C \abs*{x-x'} (1+\abs{x}^C+\abs{x'}^C)$ for some $C>0$ and all $(x,x')\in \mathbb{R}^{2d}$.
With this notation and the definition \eqref{eq : def F 3 in score},  we write
\begin{multline*}
\mathbb{F}^{\ve,(h),3}(\theta^\star)=-\frac{2}{\varepsilon}
\sum_{j=0}^{N-1} J(X^0_{j\Delta}) [R^{\ve,(h)}_{(j+1)\Delta}-R^{\ve,(h)}_{j\Delta}] 
\\+ \frac{2}{\varepsilon}
\sum_{j=0}^{N-1} [ J(X^0_{j\Delta}) - J(X^\ve_{j\Delta}) ]  [R^{\ve,(h)}_{(j+1)\Delta}-R^{\ve,(h)}_{j\Delta}] 
=: \mathbb{F}^{\ve,(h),3,1}(\theta^\star)+\mathbb{F}^{\ve,(h),3,2}(\theta^\star).
\end{multline*}
We consider the second sum, $\mathbb{F}^{\ve,(h),3,2}(\theta^\star)$, on which we use the Lipschitz property of $J$ seen above. It yields, 
\begin{multline*}
	\abs*{\mathbb{F}^{\ve,(h),3,2}(\theta^\star)}
	\le \frac{C}{\ve}  \sum_{j=0}^{N-1}\abs*{X^0_{j\Delta} - X^\ve_{j\Delta}}  \left(1+\abs*{X^0_{j\Delta}}^C + 
	\abs*{X^\ve_{j\Delta}}^C \right)  \\ \times    \left(\abs*{R^{\ve,(h)}_{(j+1)\Delta}}+\abs*{R^{\ve,(h)}_{j\Delta}}\right) ,
\end{multline*} 
with some constant $C>0$. By Lemma \ref{Lem : upper bound Rve}, Equation \eqref{Eq : moment Volterra} and Lemma \ref{Lem : conv Xve X0}, we deduce for any $p \ge 1$,
\begin{align*}
	\norm*{\mathbb{F}^{\ve,(h),3,2}(\theta^\star)}_{{\revarn\mathbf{L}}^p(\Omega)}
	&\le
\frac{C}{\ve}  \sum_{j=0}^{N-1} \ve \times C \times (\ve h^{1/2} + h^\alpha)
\\
&\le C \left(\ve N h^{1/2} + N h^\alpha\right)
\le C \left(\ve h^{-1/2} k^{-1} + h^{\alpha-1} k^{-1}\right),
\end{align*}
where in the last line we used $N =O(1/\Delta)=O(1/(kh))$. {\revarn Now, Assumption \ref{Ass: step normality} implies} that $\mathbb{F}^{\ve,(h),3,2}(\theta^\star)$ goes to zero.

For the first sum $\mathbb{F}^{\ve,(h),3,1}(\theta^\star)$, we use a discrete integration by part formula, together with $R^{\ve,(h)}_{0}=Z^{\ve,(h)}_0-Z^\ve_0=0$, to write
\begin{equation*}
\mathbb{F}^{\ve,(h),3,1}(\theta^\star)=
\frac{2}{\ve} \sum_{j=0}^{N-1}
\left(J(X_{(j+1)\Delta}^0)-J(X_{j\Delta}^0)\right) R^{\ve,(h)}_{(j+1)\Delta}
~ - \frac{2}{\ve} J(X^0_{N\Delta}) R^{\ve,(h)}_{N\Delta}.
\end{equation*}
From the Lipschitz property of $x\mapsto J(x)$ for $x$ in a compact set, and the fact that $s \mapsto X^0_s$ is bounded by a constant, we deduce
\begin{equation*}
	\mathbb{F}^{\ve,(h),3,1}(\theta^\star) \le 
	\frac{C}{\ve} \sum_{j=0}^{N-1} \abs*{X_{(j+1)\Delta}^0 - X_{j\Delta}^0} \abs*{ R^{\ve,(h)}_{(j+1)\Delta}} +  \frac{C}{\ve} \abs*{ R^{\ve,(h)}_{N\Delta}}.
\end{equation*}
Since, $s \mapsto X^0_s$ is deterministic, we can write for any $p \ge 1$,
\begin{align*}
	\norm*{ 
	\mathbb{F}^{\ve,(h),3,1}(\theta^\star)
}_{{\revarn\mathbf{L}}^p(\Omega)} &\le 
\begin{multlined}[t]
	\frac{C}{\ve} \sum_{j=0}^{N-1} \abs*{X_{(j+1)\Delta}^0 - X_{j\Delta}^0}
	 \norm{ R^{\ve,(h)}_{(j+1)\Delta}}_{{\revarn\mathbf{L}}^p(\Omega)}  
	 \\+  \frac{C}{\ve} \norm{ R^{\ve,(h)}_{N\Delta}}_{{\revarn\mathbf{L}}^p(\Omega)} 
	 \end{multlined}
	 \\
	 &\le 	\frac{C}{\ve} \sup_{j=1,\dots,N}  \norm{ R^{\ve,(h)}_{j\Delta}}_{{\revarn\mathbf{L}}^p(\Omega)}, \text{ using Lemma \ref{lem : X^0 bounded var},}
	 \\ 
	 &\le 	C \left(h^{1/2} + \ve^{-1}h^\alpha\right)
	 \text{ by Lemma \ref{Lem : upper bound Rve}.}
\end{align*}
As $h\to0$ and $ h^{\alpha}=o(\Delta^\alpha) = o(\ve)$ from the condition $kh = o(\ve^{1/\alpha})$ in Assumption \ref{Ass: step normality}, we deduce that $	\mathbb{F}^{\ve,(h),3,1}(\theta^\star)$ converges to zero.}
%
\end{proof}
\begin{lemma} \label{lem : conv partial2 C}
		Assume \ref{Ass : global Lip}, {\revK \ref{Ass : kernel resolvent},} \ref{Ass : K mino}, \ref{Ass : ident}, \ref{Ass: step consistency}, and that $b\in\mathcal{C}^{2}_{\mathcal{P}}(\mathbb{R}^d \times \overset{\circ}{\Theta}) \cap \mathcal{C}^{0,3}_{\mathcal{P}}(\mathbb{R}^d \times \overset{\circ}{\Theta})$. Then, for all $(u,v) \in \{1,\dots,d_\Theta \}^2$, we have
		\begin{equation*}
	\sup_{s\in[0,1]}	\abs*{	 \frac{\partial^{2}}{\partial \theta_u {\revarn\partial}\theta_v}
			\frac{1}{\Delta}\mathbb{C}^{\ve,(h)}(\theta^\star+s( \widehat{\theta}_\ve - \theta )) - 2I_{u,v}  }  \xrightarrow[\P]{\ve\to0}0.
		\end{equation*}		
\end{lemma}
	\begin{proof}
		From \eqref{Eq : def Xi} and \eqref{Eq : partial C},  we have for all $\theta \in \overset{\circ}{\Theta}$, 
		$	\frac{\partial^{2}}{\partial \theta_u \partial \theta_u}
		\frac{1}{\Delta}\mathbb{C}^{\ve,(h)}(\theta) = \mathbb{G}^{\ve,(h),1}(\theta)+\mathbb{G}^{\ve,(h),2}(\theta)$, where
		\begin{align*}
		\mathbb{G}^{\ve,(h),1}(\theta)&:=2 \Delta\sum_{j=0}^{N-1}
		\frac{\partial}{\partial \theta_u} b(X^\ve_{j\Delta},\theta)^{{\revarn *}} H(X^\ve_{j\Delta}) 
		\frac{\partial}{\partial \theta_v} b(X^\ve_{j\Delta},\theta)
		\\
		\mathbb{G}^{\ve,(h),2}(\theta)&:=-2 \sum_{j=0}^{N-1} 
		 \frac{\partial^{2}}{\partial \theta_u {\revarn\partial}\theta_v} b(X^{\ve}_{j\Delta},\theta)^{{\revarn *}} 
		 H(X^\ve_{j\Delta}) \Xi_{{\revarn j}}^{\ve,(h)}(\theta)		
		\end{align*}
Using that  $x \mapsto \partial_\theta b(x,\theta)$ is of class $\mathcal{C}^1$ with derivative having at most polynomial growth and the Lipschitz property of $H$, it is possible to get that
\begin{equation*}
	\E\left[ \sup_{\theta\in\overset{\circ}{\Theta}} \abs*{ 
			\mathbb{G}^{\ve,(h),1}(\theta) - \mathbb{G}^{0,(h),1}(\theta) } \right] \le C
			 \Delta \sum_{j=0}^{N-1} 	\E\left[ \abs*{X^\ve_{j\Delta}-X^0_{j\Delta}}^2\right]^{1/2}.
\end{equation*}		
		From Lemma \ref{Lem : conv Xve X0}, we deduce that 
		\begin{equation} \label{Eq : conv D2 G ve 0}
			\E\left[ \sup_{\theta\in\overset{\circ}{\Theta}} \abs*{ 
				\mathbb{G}^{\ve,(h),1}(\theta) - \mathbb{G}^{0,(h),1}(\theta) } \right] \xrightarrow{\ve\to0}0.
		\end{equation}
		From the continuity of the deterministic path $t \mapsto X^0_t$ for $t\in [0,T]$, we deduce that
		 the Riemann sum $	\mathbb{G}^{0,(h),1}(\theta)$ is converging uniformly for $\theta\in \Theta$ to $2\int_0^T 
			\frac{\partial}{\partial \theta_u} b(X^0_t,\theta)^{{\revarn *}} H(X^0_t) 
		\frac{\partial}{\partial \theta_v} b(X^0_t,\theta)dt$. The consistency of the estimator ${\revarn \widehat{\theta}_\ve}$ entails,
\begin{equation*}
	\sup_{s\in[0,1]}	\abs*{	\mathbb{G}^{0,(h),1}(\theta^\star+s( \widehat{\theta}_\ve - \theta^{{\revarn \star}} )) - 2I_{u,v}  }  \xrightarrow[\P]{\ve\to0}0.
	\end{equation*}		
Recalling \eqref{Eq : conv D2 G ve 0}, we see that the lemma will be proved as soon as we show
 $	\sup_{s\in[0,1]}	\abs*{\mathbb{G}^{\ve,(h),2} (\theta^\star+s( \widehat{\theta}_\ve - \theta^{{\revarn \star}} ))} \xrightarrow[\mathbb{P}]{\ve \to 0}0$.	
 By \eqref{Eq : Xi expanded}, we write $\mathbb{G}^{\ve,(h),2}(\theta)=\mathbb{G}^{\ve,(h),2,1}(\theta)+\mathbb{G}^{\ve,(h),2,2}(\theta)$ where
\begin{align*}
	\mathbb{G}^{\ve,(h),2,1}(\theta)&:=-2 \sum_{j=0}^{N-1} 
	\frac{\partial^{2}}{\partial \theta_u {\revarn\partial}\theta_v} b(X^{\ve}_{j\Delta},\theta)^{{\revarn *}}
	H(X^\ve_{j\Delta}) \ve \int_{j\Delta}^{(j+1)\Delta} a(X^\ve_t) dB_t \\	
	\mathbb{G}^{\ve,(h),2,2}(\theta)&:=
	\begin{multlined}[t]
			-2 \sum_{j=0}^{N-1} 
\frac{\partial^{2}}{\partial \theta_u {\revarn\partial}\theta_v} b(X^{\ve}_{j\Delta},\theta)^{{\revarn *}} 
H(X^\ve_{j\Delta}) \big[\int_{j\Delta}^{(j+1)\Delta} \big(b(X_t^\ve,\theta^\star) \\- b(X_{j\Delta}^\ve,\theta)\big) dt
+ R^{\ve,(h)}_{(j+1)\Delta}-R^{\ve,(h)}_{j\Delta}\Big]	
\end{multlined}
\end{align*}

As $	\mathbb{G}^{\ve,(h),2,1}(\theta)$ can be seen as a stochastic integral on $[0,T]$ driven by a {\revarn Brownian motion,} we can use {\revarn the
Burkholder-Davis-Gundy inequality,} with the fact that $b \in \mathcal{C}^{0,3}_\mathcal{P}$ to derive that for any $q \ge 1$,
\begin{align*}
&	\E \left[\abs*{\mathbb{G}^{\ve,(h),2,1}(\theta)}^q\right] \le \ve^q c(q), \quad \forall \theta \in \overset{\circ}{\Theta} \\
&	\E \left[\abs*{\mathbb{G}^{\ve,(h),2,1}(\theta)-\mathbb{G}^{\ve,(h),2,1}(\theta')}^q\right] \le \ve^q c(q)
	\abs{\theta-\theta'}^q, \quad \forall \theta,\theta' \in \overset{\circ}{\Theta}
\end{align*}
		It is sufficient to deduce that 
		 $\sup_{\theta \in \overset{\circ}{\Theta}} \abs*{\mathbb{G}^{\ve,(h),2,1}(\theta)} \xrightarrow[\mathbb{P}]{\ve \to 0}0$.
	
Finally, we focus on 	$\mathbb{G}^{\ve,(h),2,2}$. Using the sub-polynomial growth of the derivatives of $b$ with Lipschitz property of $H$, we deduce
\begin{align*}
\sup_{s \in [0,1]}	&\abs*{\mathbb{G}^{\ve,(h),2} (\theta^\star+s( \widehat{\theta}_\ve - \theta ))} \le
	\\
	& c \sum_{j=0}^{N-1} (1+\abs{X_{j\Delta}^\ve}^c) \int_{j\Delta}^{(j+1)\Delta} \abs*{ b(X^\ve_t,\theta^\star)-b(X^\ve_{j\Delta},\theta^\star)  }dt 
	\\ {\revarn +}
	& c \Delta \sum_{j=0}^{N-1} (1+\abs{X_{j\Delta}^\ve}^c) \abs*{\widehat{\theta}_\ve - \theta^\star}
	\\ {\revarn +}
	& c \sum_{j=0}^{N-1} (1+\abs{X_{j\Delta}^\ve}^c)\abs*{[R^{\ve,(h)}_{(j+1)\Delta}-R^{\ve,(h)}_{j\Delta}}.
\end{align*}
	The first and third terms in sum above can be shown to go to zero by computations similar to the study of the terms $	\mathbb{F}^{\ve,(h),2}$ and {\modar $	\mathbb{F}^{\ve,(h),3,1}$} in the proof of Lemma \ref{lem : conv quasi score}. Let us stress that the convergence is faster here due the missing $\frac{1}{\ve}$ renormalisation factor which appears in Lemma \ref{lem : conv quasi score}.	
	The second sum above goes to zero in probability due to the consistency of the estimator.	
\end{proof}

\subsubsection{Proof of Theorem \ref{Thm : normality}}

Remark that Assumption \ref{Ass: step normality} implies Assumption \ref{Ass: step consistency}. 

Then, by Taylor expansion of the contrast function around $\widehat{\theta}_\ve$, we have
\begin{align*}
	0&=\partial_\theta \mathbb{C}^{\ve,(h)}(\widehat{\theta}_\ve)
\\	&=\partial_\theta \mathbb{C}^{\ve,(h)}(\theta^\star) +
	\left(\int_0^1 D^2_\theta \mathbb{C}^{\ve,(h)}( \theta^\star+ s(\widehat{\theta}_\ve-\theta^\star) ) 
	 ds \right) {\revarn (\widehat{\theta}_\ve-\theta^\star),}
\end{align*}
 where $D^2_\theta \mathbb{C}^{\ve,(h)}$ is the matrix with entries 
 $\frac{\partial^{2}}{\partial \theta_u \theta_v}\mathbb{C}^{\ve,(h)}$ for $1\le u,v \le d_\Theta$.
 Now, the theorem is a consequences of Lemmas \ref{lem : conv quasi score}--\ref{lem : conv partial2 C} with Assumption \ref{Ass: Fisher inver}. 
\qed

\section{Numerical simulation}\label{S : num sim}

In this section, we explore by a Monte Carlo experiment the behaviour of the estimator on a finite sample. In particular, we focus on the effect of the choice of the tuning parameter $k$.

We consider a linear model
\begin{equation*}
	X^\varepsilon_t=X^\varepsilon_0+\ve \int_0^t K(t-s) dB_s
	+ \int_0^t K(t-s) (\theta_0 X_s^\ve+\theta_1) ds.
\end{equation*}
From the simple structure of this model, the estimator is explicitly given by solution of a $2\times2$ system,
\begin{equation} \label{Eq : est linear}
	\begin{bmatrix}
	\Delta \sum_{j=0}^{N-1} (X^\ve_{j\Delta})^2 &
	\Delta \sum_{j=0}^{N-1} X^\ve_{j\Delta} \\
	\Delta \sum_{j=0}^{N-1} X^\ve_{j\Delta} &
	N \Delta  
\end{bmatrix}
\begin{bmatrix}
\widehat{\theta}_{0,\ve}
	\\
\widehat{\theta}_{1,\ve}
\end{bmatrix}
=
\begin{bmatrix}
   \sum_{j=0}^{N-1} \big(Z^{\ve,(h)}_{(j+1)\Delta} - Z^{\ve,(h)}_{j\Delta} \big)  X^\ve_{j\Delta}
\\
	\sum_{j=0}^{N-1}  \big(Z^{\ve,(h)}_{(j+1)\Delta} - Z^{\ve,(h)}_{j\Delta} \big)
\end{bmatrix}.	
\end{equation}
The true value of the parameter is $(\theta_0^\star,\theta_1^\star)=(-1,1)$, and we set $\alpha=0.8$. {\revarn The sampling step on $X^\ve$ is $h=10^{-2}$.}
Tables \ref{T : T 1 n 100New}--\ref{T : T 50 n 5000New} correspond to the different horizons $T=1$, $T=10$ and $T=50$. These tables display the empirical mean, and the empirical standard deviation, rescaled by the rate $\ve^{-1}$, for the estimator $\widehat{\theta}_\ve$. 
{\revarn The tables also show results for $\widetilde{\theta}_\ve$ which is computed by using the same system as \eqref{Eq : est linear} but where values of the reconstructed semimartingale
	$(Z^{\ve,(h)}_{i\Delta})_i$ are replaced by the true values of semimartingal $(Z^\ve_{i\Delta})_i$. 
	Since $Z^\ve$ is not directly observed, the quantity $\widetilde{\theta}_\ve$ is not a feasible estimator. However, comparing $\widehat{\theta}_\ve$ and $\widetilde{\theta}_\ve$ allows us to assess the effect of the reconstruction step on the estimation.}
We used $10^3$ replications to compute these empirical quantities.
 The results for different choices of the tuning parameter $k$ are shown. 
 {\revarn The estimator performs well as soon as $\ve$ is small enough, as predicted by the theory. Larger values of $T$ improve the performance of the estimator.  	
 	We observe that the choice of $k$ affects the bias of the estimation, and the optimal choice depends on the values of $\ve$. In our setting, the optimal value of $k$ tends to be smaller, when $\ve$ is small. 
 	In few situations with small $k$, the difference between $\widehat{\theta}_\ve$ and $\widetilde{\theta}_\ve$ is noticeable, indicating that the reconstruction error of the semimartingal is not negligible. If $k$ is large, the estimators $\widehat{\theta}_\ve$ and $\widetilde{\theta}_\ve$ are very close, however, if $k$ is too large, they both suffer from a bias due to the sampling with step $\Delta$.} 
 Let us stress that the condition 
{\modar $h=o(\varepsilon^{1/\alpha^2})$} from Assumption \ref{Ass: step normality} is not satisfied in all the numerical experiments presented here. 
Indeed, for $\varepsilon=10^{-2}$ it would require that the step $h$ is negligible versus $10^{-2/0.8^2}\approx 7.5* 10^{-4}$. Despite this, we observe that the estimator performs well for much larger values of $h$. 
This fast sampling condition ensures that the three conditions on the choice of $k$ in Assumption \ref{Ass: step normality} are compatible. The first condition can be written $\Delta^{\alpha}=o(\ve^{-1})$ and imposes that $k$ is small. This condition on $\Delta$ 
appears in the proof of Lemma \ref{lem : conv quasi score} to reduce a bias arising from the discretization with step $\Delta$ of $X^\ve$. 
 The last two conditions impose that $k$ is large enough, in order to reduce the cumulative error coming for  the reconstruction of the path of $Z^\ve$ from the observation of the rough process $X^\ve$. 
 {\revarn When $\ve=10^{-2}$ it is impossible to choose $k$ such that the three conditions of 
 	\ref{Ass: step normality} are valid. Nevertheless, our simulations show that the best results are obtained when $k$ is small, and are very correct. On the other hand, choosing $\Delta$ too large consistently introduces bias.}
 
 Comparing with the behaviour of the Trajectory Fitting Estimator studied in \cite{gloterDriftEstimationRough2025}, we see that the TFE estimator exhibits a {\revarn more stable
 behaviour. Its performance is almost independent of the sampling step, and its implementation does not require to choose a tuning parameter like $k$.} 
On the other hand, the asymptotic variance of the QMLE estimator {\revarn is} smaller, in particular for large values of $T$. {\revarn This is consistent with the efficiency of the QMLE as discussed in Remark \ref{R : Fisher}.}

{\modar Tables \ref{T : T 1 n 100 alpha095New}--\ref{T : T 50 n 5000 alpha095New} display the results for $\alpha=0.95$. In this smoother case, the choice of $k$ appears to be a less critical than when $\alpha=0.8$. In contrast, results for $\alpha=0.6$, which corresponds to a rougher situation encountered in finance, show less stability of the estimator with respect to the choice of $k$ (see Tables  \ref{T : T 1 n 100 alpha06New}--\ref{T : T 50 n 5000 alpha060New}).}


\section{Appendix} \label{S : Appendix}
{\revarn We begin the Appendix by establishing a convenient} extension of the ${\revarn\mathbf{L}}^p$ control on increments of Volterra {\revarn processes} given in 
Proposition 4.1 of  \cite{richardDiscretetimeSimulationStochastic2021}.


\begin{lemma} \label{Lem : cond moment process X}
	Assume \ref{Ass : global Lip}.
	For any $p \ge 1$, there exists two collections of random variables $(A_{s,t}^{\ve}(p))_{0\le s < t \le T}$ and $(B_{s,t}^{\ve}(p))_{0\le s < t \le T}$ such that, for all $s$, $(\omega,t) \mapsto A_{s,t}^{\ve}(p)(\omega)$ and $(\omega,t) \mapsto B_{s,t}^{\ve}(p)(\omega)$ are measurable with respect to $\mathcal{F}\otimes\mathcal{B}((s,T])$ and such that:

	1) For all $0 \le s < t \le T$, the random variables $A_{s,t}^\ve$ and $B_{s,t}^\ve$ are $\mathcal{F}_s$-measurable.
	
	2)For all $p\ge 1$, there exists $c(p) \ge 0$, for all $0\le s < t\le T$, $\ve\in(0,1]$,
	\begin{equation} \label{Eq : condit moment increment Xve}
		\E \left[\abs{X^\ve_t-X^\ve_s}^p\mid \mathcal{F}_s\right] \le c(p) \left[
		\ve^{p} (t-s)^{(\alpha-1/2) p} A^\ve_{s,t}(p) + (t-s)^{\alpha p}  {B_{s,t}^\ve(p)}
		\right].
	\end{equation} 
	
	3) For all $p\ge 1$, $q\ge1$, there exists $c(p,q) \ge 0$ such that
	\begin{equation} \label{Eq : Lq finite Ast and Bst}
		\E\left[\abs*{A_{s,t}^\ve(p)}^q\right] \le c(p,q), \quad
		\E\left[\abs*{B_{s,t}^\ve(p)}^q\right] \le c(p,q), 
	\end{equation}
	for all $0\le s < t \le T$, $0<\ve\le 1$.
\end{lemma}
\begin{proof}
	For $s\in [0,T]$ and $p \ge 1$, we denote ${\revarn U^\ve_{s,t}(p)}:=\E\left[\abs*{X^\ve_t}^p  \mid \mathcal{F}_s\right]$. As we know that $\sup_t \E\left[ \abs*{X^\ve_t}^q \right] \le c(q)$ for all $q$, and that $X^\ve$ admits continuous paths, we deduce that $(\omega,t) \mapsto {\revarn U^\ve_{s,t}(p)}(\omega)$ is measurable. 
	We also deduce 
	\begin{equation} \label{Eq : upper bound Zst}
		\sup_{t\in[s,T]} {\revarn\E}\left[\abs*{{\revarn U^\ve_{s,t}(p)}}^q\right] \le c(p,q).
	\end{equation}
	%
	%
	%
	%
	%
	%
	%
	%
	By \eqref{Eq : Volterra SDE}, we write for $t>s$,
	$X^\ve_{t}-X^\ve_s=\sum_{l=1}^4 D^{\ve,l}_{s,t}$, where
	\begin{align*}
		&	D^{\ve,1}_{s,t}=\int_s^t K(t-u)b(X^\ve_u,\theta^\star) du, \quad  
		D^{\ve,2}_{s,t}=\ve \int_s^t K(t-u)a(X^\ve_u) dB_u,
		\\
		&	D^{\ve,3}_{s,t}=\int_0^s \left[ K(t-u)-K(s-u) \right]b(X^\ve_u,\theta^\star) du , 
		\\
		& 	D^{\ve,4}_{s,t}=\ve \int_0^s \left[ K(t-u)-K(s-u) \right] a(X^\ve_u) dB_u  .
	\end{align*}
	Let us upper bound the conditional moments {\revarn of} these quantities to {\revarn identify} the candidates for the variables 
	$A^{\ve}_{s,t}(p)$ and $B^{\ve}_{s,t}(p)$.
	
	For $D^{\ve,1}_{s,t}$, we use Jensen's inequality, and {\revK $\abs{K}\le c K_0$} to obtain
	\begin{multline*}
		{\revarn\E}\left[\abs*{D^{\ve,1}_{s,t}}^{{\revarn p}} \mid \mathcal{F}_s\right]	= E\left[ \abs*{\int_s^t K(t-u)b(X^\ve_u,\theta^\star) du}^p \mid  \mathcal{F}_s\right]
		 \\ \le  {\revK c} \left(\int_s^t K_{{\revK 0}}(t-u) du\right)^{p-1}  {\revarn\E}\left[ \int_s^t K_{{\revK 0}}(t-u)\abs*{b(X^\ve_u,\theta^\star)}^p du \mid  \mathcal{F}_s\right]
		\\ \le c (t-s)^{\alpha (p-1)} \int_s^t K_{{\revK 0}}(t-u) {\revarn\E}\left[  \abs*{b(X^\ve_u,\theta^\star)}^p \mid \mathcal{F}_s\right] du.
	\end{multline*}
	Now, we use $\abs{b(X^\ve_u,\theta^\star)}^p\le c(1+\abs*{X^\ve_u}^p)$, ${\revarn\E}\left[ 
	\abs{X^\ve_u}^p	\mid \mathcal{F}_s\right] = {\revarn U^\ve_{s,u}(p)}$ {\revK and the expression of $K_0$,} to deduce,
	\begin{equation*}
		{\revarn\E}\left[ \abs[\Big]{\int_s^t K_{{\revK 0}}(t-u)b(X^\ve_u,\theta^\star) du}^p \mid  \mathcal{F}_s\right]
		\le c (t-s)^{\alpha (p-1)}  \int_s^t (t-u)^{\alpha-1}\left[1+\abs*{{\revarn U^\ve_{s,u}(p)}}\right] du
	\end{equation*}
	We set $B^{\ve,1}_{s,t}(p)=(t-s)^{-\alpha}\int_s^t (t-{\revarn u})^{\alpha-1}\left[1+\abs*{{\revarn U^\ve_{s,u}(p)}}\right] du$. Then, we have obtained
	\begin{equation}\label{Eq : maj eq cond D1 st}
		{\revarn\E}\left[\abs*{D^{\ve,1}_{s,t}}^{{\revarn p}}  \mid  \mathcal{F}_s\right]
		\le c (t-s)^{\alpha p  }B^{\ve,1}_{s,t}(p).
	\end{equation}
	Remark that $(t-s)^{-\alpha}(t-u)^{\alpha-1}\one_{[s,t)}(u) du$ is a finite mass measure. By the triangle inequality, we deduce
	\begin{multline} \label{Eq : majo moment q B1st}
		\norm*{ B^{\ve,1}_{s,t}(p)  }_{{\revarn\mathbf{L}}^q(\Omega)} \le 
		(t-s)^{-\alpha}\int_s^t (t-u)^{\alpha-1}\norm{ 1+ \abs*{   {\revarn U^\ve_{s,u}(p)}}}_{{\revarn\mathbf{L}}^q(\Omega)}
		du	\\\le c (t-s)^{-\alpha}\int_s^t (t-u)^{\alpha-1}c(p,q) du\le c(p,q),
	\end{multline}
	where we used \eqref{Eq : upper bound Zst}.
	
	Concerning $D^{\ve,2}_{s,t}$, we use {\revarn BDG} inequality to obtain,
	\begin{align*}
		\E\left[\abs*{D^{\ve,2}_{s,t}}^{{\revarn}p} \mid \mathcal{F}_s \right] 
		&=\ve^p \E\left[ \abs*{\int_s^t K(t-u)a(X^\ve_u) dB_u}^p \mid  \mathcal{F}_s\right] 
		\\ &\le c(p)   \ve^p \E\left[ \abs*{\int_s^t K_{{\revK 0}}(t-u)^2\abs{a(X^\ve_u)}^2 du}^{p/2}\mid\mathcal{F}_s \right] 
		\\ &\le c(p)   \ve^p \left(\int_s^t K_{{\revK 0}}(t-u)^2du \right)^{p/2-1}{\revarn\E}\left[ \int_s^t K_{{\revK 0}}(t-u)^2\abs*{a(X^\ve_u)}^p du\mid\mathcal{F}_s \right] ,
		\\ &\le c(p)   \ve^p (t-s)^{(2\alpha-1)(p/2-1)}\E\left[ \int_s^t (t-u)^{2\alpha-2}\abs*{a(X^\ve_u)}^p du\mid\mathcal{F}_s \right] ,
	\end{align*}
	where we used {\revK $\abs*{K}\le cK_0$}, the Jensen's inequality to get the third line, 
	 and
	the definition of $K_{{\revK 0}}$ in the last one. By Assumption \ref{Ass : global Lip} and {\revarn the} definition of ${\revarn U^\ve_{s,u}(p)}$, it entails
	\begin{equation*}
		E\left[\abs*{D^{\ve,2}_{s,t}}^{{\revarn p}} \mid \mathcal{F}_s \right] \le 
		c(p)   \ve^p (t-s)^{(2\alpha-1)(p/2-1)}
		\int_s^t (t-u)^{2\alpha-2} \left[1+\abs*{{\revarn U^\ve_{s,u}(p)}}\right] du. 
	\end{equation*}
	We set $A^{\ve,1}_{s,t}(p) := (t-s)^{1-2\alpha} \int_s^t 
	(t-u)^{2\alpha-2} \left[1+\abs*{{\revarn U^\ve_{s,u}(p)}}\right] du$. With this notation, we have shown,
{\revarn	\begin{equation}\label{Eq : maj eq cond D2 st}
	E\left[\abs*{D^{\ve,2}_{s,t}}^{{\revarn p}} \mid \mathcal{F}_s \right] \le 
	c(p)   \ve^p (t-s)^{(\alpha-1/2)p}  A^{\ve,1}_{s,t}(p).
\end{equation}}
	Moreover, using the triangle inequality 
	and \eqref{Eq : upper bound Zst}, we have {\revarn for $q \ge 1$,}
	\begin{multline}\label{Eq : majo moment q A1st}
		\norm*{ A^{\ve,1}_{s,t}(p)  }_{{\revarn\mathbf{L}}^q(\Omega)} \le 
		(t-s)^{1-2\alpha}\int_s^t (t-u)^{2\alpha-2}\norm*{ 1+ \abs*{ {\revarn U^\ve_{s,u}(p)}}}_{{\revarn\mathbf{L}}^q(\Omega)}
		du	\\\le c (t-s)^{1-2\alpha}\int_s^t (t-u)^{2\alpha-2}c(p,q) du\le c(p,q).
	\end{multline}
	
	Now, we define for $0\le s <t \le {\revarn T}$, $A^{\ve}_{s,t}(p):=A^{\ve,1}_{s,t}(p)+
	{\revarn \ve^{-p}}
	(t-s)^{{\revarn (1/2-\alpha)p}}\abs*{D^{\ve,4}_{s,t}}^p$ and 
	$B^{\ve}_{s,t}(p):=B^{\ve,1}_{s,t}(p)+(t-s)^{-\alpha p}\abs*{D^{\ve,3}_{s,t}}^p$.
	Since $D^{\ve,3}_{s,t}$ and $D^{\ve,4}_{s,t}$ are $\mathcal{F}_s$-measurable variables, and collecting $X^\ve_{t}-X^\ve_s=\sum_{l=1}^4 D^{\ve,l}_{s,t}$ with \eqref{Eq : maj eq cond D1 st}, \eqref{Eq : maj eq cond D2 st}, we have
	$\E \left[\abs{X^\ve_t-X^\ve_s}^p\mid \mathcal{F}_s\right] \le c(p) 
	\bigl[  (t-s)^{\alpha p  }B^{\ve,1}_{s,t}(p) + \ve^p (t-s)^{(\alpha-1/2)p}  A^{\ve,1}_{s,t}(p) 
	+ \abs[\big]{D^{\ve,3}_{s,t}}^p+\abs[\big]{D^{\ve,4}_{s,t}}^p\bigr]=c(p) 
	\bigl[  (t-s)^{\alpha p  }  B^{\ve}_{s,t}(p)+  \ve^p (t-s)^{(\alpha-1/2)p} A^{\ve}_{s,t}(p)\bigr]$. This proves \eqref{Eq : condit moment increment Xve}.
	
	To end the proof of the lemma, it remains to check the validity of \eqref{Eq : Lq finite Ast and Bst}. As we already have shown upper bounds of ${\revarn\mathbf{L}}^q(\Omega)$ norms of $A^{\ve,1}_{s,t}(p)$ and $B^{\ve,1}_{s,t}(p)$
	in \eqref{Eq : majo moment q A1st} and \eqref{Eq : majo moment q B1st}, it remains to 
	see that 
	\begin{align} \label{Eq : control moment r D3 D4}
		\E\left[ \abs*{D^{\ve,3}_{s,t}}^r \right] \le c(r)(t-s)^{\alpha r} ~ \text{and,} \quad
		\E\left[ \abs*{D^{\ve,4}_{s,t}}^r \right] \le c(r)\ve^{\revarn r} (t-s)^{(\alpha-1/2)r},
	\end{align}
	where $r=pq$.
Using Jensen's inequality, we have
\begin{multline*}
	\E\left[ \abs*{D^{\ve,3}_{s,t}}^r \right]
	\le \left( \int_0^s \abs{K(t-u)-K(s-u)} du \right)^{r-1} \\
	{\revarn \times} \E\left[ \int_0^s
	\abs*{K(t-u)-K(s-u)} \abs*{b(X^\ve_u,\theta)}^rdu \right]
\end{multline*} 
Using \eqref{Eq : moment Volterra}, we deduce 
$\E\left[ \abs*{D^{\ve,3}_{s,t}}^r \right] \le \left( \int_0^s \abs{K(t-s)-K(s-u)} du \right)^{r}$. 
Now, we use that, {\revK by $\abs*{K'}\le c \abs*{K_0'}$, we have} $ \int_0^s \abs{K(t-u)-K(s-u)} du = O\big((t-s)^\alpha\big)$,  {\revK and} deduce
$\E\left[ \abs*{D^{\ve,3}_{s,t}}^r \right] \le C (t-s)^{\alpha r}$. It is the first part 
of \eqref{Eq : control moment r D3 D4}.
For the second part, we use successively BDG inequality and Jensen's inequality, to get 
\begin{multline*}
	\E\left[ \abs*{D^{\ve,4}_{s,t}}^r \right] 
	\le c(r) {\revarn \ve^r}	\E\left[ \abs*{\int_0^s( K(t-u)-K(s-u))^2 \abs*{a(X^\ve)}^2du}^{r/2}   \right]
	\le \\ 
	\le c(r) \left( \int_0^s (K(t-u)-K(s-u))^2 du \right)^{r/2-1} \\ \times \E\left[ \int_0^s
	(K(t-s)-K(s-u))^2 \abs{a(X^\ve_u,\theta)}^r du \right].
\end{multline*} 
Using \eqref{Eq : moment Volterra}, we deduce $\E\left[ \abs*{D^{\ve,4}_{s,t}}^r \right]  \le {\revarn c(r) \ve^r} \left( \int_0^s (K(t-u)-K(s-u))^2 du \right)^{r/2} \le {\revarn c(r) \ve^r} (t-s)^{(\alpha/2-1)r}$ as
$ \int_0^s (K(t-u)-K(s-u))^2 du = O\big((t-s)^{2\alpha-1}\big)$. This proves \eqref{Eq : control moment r D3 D4} and terminates the proof of the lemma.
\end{proof}
{\revK
\begin{lemma}\label{lem : X^0 bounded var}
	Assume Assumptions \ref{Ass : global Lip} and that $b\in \mathcal{C}^{2,0}_{\mathcal{P}}(\mathbb{R}^d \times \overset{\circ}{\Theta})$. Then, the function $ \left\{\begin{array}{rl}
	 [0,T] &\to \mathbb{R}^d\\
	s&\mapsto X^0_s\end{array}\right.$ has bounded variation.	
\end{lemma}
\begin{proof}
	We apply Theorem 1 in \cite{millerSmoothnessSolutionsVolterra1971} and deduce that the solution $X^0$ of the deterministic Volterra equation 
	\eqref{Eq : Volterra SDE} with $\ve=0$ is $\mathcal{C}^1$ on $(0,\infty)$. Remark that \eqref{Eq : comp K K0}, with the fact that $u \mapsto K_0(u)$ is nonincreasing, implies Assumption (A3) of \cite{millerSmoothnessSolutionsVolterra1971}. 
	Also, as noted in the beginning of the proof of Theorem 1 in \cite{millerSmoothnessSolutionsVolterra1971}, the function
	$t \mapsto \frac{\partial X^0_t}{\partial t}$ is in $\mathbf{L}^{1}((0,T])$. It follows that $s\mapsto X^0_s$ has finite variation on $[0,T]$.
%
\end{proof}
}
{\revK
\begin{proposition} \label{Prop : suff cond for A2}
	Assume that the kernel $K$ is given by $K=K_0+K_0 \star G$ for  
	$G: (0,\infty) \to \mathbb{R}$
	a locally bounded measurable function. Then,  Assumption \ref{Ass : kernel resolvent} is valid.	
\end{proposition}
\begin{proof}
	Recall that $L_0$ is the first kind resolvent of $K_0$. We consider the following functional linear equation in $L$, 
		\begin{equation*}
		L=L_0-G\star L.
	\end{equation*}
	As $L_0$ and $G$ are locally integrable, we know by Theorem 3.5 in \cite{gripenbergVolterraEquationsFirst1980} that the equation admits a solution $L : (0,\infty) \to \mathbb{R}$ which is locally integrable. Moreover, since $G$ is locally bounded and $L$ locally integrable, we see that for all $T\ge 0$,
	$\norm*{G \star L}_{\mathbf{L}^\infty((0,T))} \le \norm*{G}_{\mathbf{L}^\infty((0,T))} \norm*{G}_{\mathbf{L}^1((0,T))}<\infty$. Hence $L-L_0$ is locally bounded, which is sufficient to deduce \eqref{Eq : comp L L0}. It remains to check that $L$ is the first kind resolvent of $K$. We write
	$L \star K= L \star (K_0 + K_0 \star G)= L\star K_0 + L \star K_0 \star G = (L_0 - G \star L) \star K_0 +  L \star K_0 \star G
	=L_0 \star K_0 -  G \star L \star K_0 +  L \star K_0 \star G =L_0 \star K_0 =    1$.
\end{proof}

}

%
%

\begin{table}[p]\caption{$\alpha=0.8$, $(\theta_0,\theta_1)=(-1,1)$, $T=1$, $h=10^{-2}$}  \label{T : T 1 n 100New}
	\centering \footnotesize
	\begin{tblr}{
			vline{1,2,3,5,7,9} = {1-12}{},
			vline{4,6,8} = {3-12}{},
			cell{2}{3-8} = {c=1}{c},
			hline{1,3,5,7,9,11,13} = {-}{},	
		}
		\SetCell[c=1,r=2]{c}\backslashbox{$\Delta$}{$\ve$}	&  & \SetCell[c=2]{c}1/10 & & \SetCell[c=2]{c}1/20 & & \SetCell[c=2]{c} 1/100 &\\
		& & $\widehat{\theta}_\ve$ & $\widetilde{\theta}_\ve$ 
		& $\widehat{\theta}_\ve$ & $\widetilde{\theta}_\ve$ & $\widehat{\theta}_\ve$ & $\widetilde{\theta}_\ve$ \\
		1/5   & mean        &  (-0.85, 0.89) & (-0.82, 0.88) &  (-0.80, 0.88)& (-0.78, 0.86)   &  (-0.80, 0.86) & (-0.77, 0.86)     \\
		(k=20) &  resc. std.$^*$   &  (5.3, 1.7) &  (5.3, 1.7)  &  (5.4, 1.8)& (5.3, 1.7)    &  (5.4, 1.8)  & (5.4, 1.8)   \\
		1/10   & mean        &  (-1.07, 0.99) & (-1.02, 0.98) &  (-0.96, 0.96)& (-0.93, 0.94)   &  (-0.91, 0.94) & (-0.88, 0.93)   \\
		(k=10) &  resc. std.    &  (5.7, 2.2) &   (5.6, 2.2) &  (5.9, 2.2)& (5.8, 2.2)   &  (5.9, 2.3) & (5.8, 2.2)  \\
		1/20  &  mean       & (-1.20, 1.07) & (-1.12, 1.03)  & (-1.03, 1.00)&  (-0.99, 0.98)   & (-0.97, 0.98) &  (-0.94, 0.96)   \\
		(k=5)	&  resc. std.  & (6.1, 2.5) & (5.9, 2.4)  &  (6.2, 2.6) & (6.1, 2.5)    &  (5.8, 2.3)  & (5.7, 2.3) \\
		1/50  &mean      & (-1.33, 1.13) & (-1.21, 1.07)  &  (-1.08, 1.03)& (-1.02, 1.00)   &  (-1.02, 1.00)   & (-0.98, 0.99) \\
		(k=2)
		&  resc. std.    &   (6.6, 2.7) & (6.3, 2.6)  &  (6.5, 2.7)&  (6.3, 2.6)  &   (6.3, 2.6) & (6.2, 2.6)\\
		1/100 & mean       &  (-1.40, 1.16) &(-1.20, 1.09)  &  (-1.14, 1.06)&  (-1.06, 1.02)   & (-1.03, 1.01)&  (-0.99, 0.99)    \\
		(k=1)  & resc. std.     &  (6.8, 2.9) & (6.4, 2.7)  &  (6.3, 2.7)&(6.1, 2.6)    &    (6.4, 2.7) & (6.3, 2.7)
	\end{tblr}\\
	\footnotesize{$^*$ resc. std. is the value of the empirical standard deviation multiplied by $1/\ve$.}
\end{table}

%
%

\begin{table}[p]\caption{$\alpha=0.8$, $(\theta_0,\theta_1)=(-1,1)$, $T=10$, $h=10^{-2}$} \label{T : T 10 n 1000New}
	\centering \footnotesize
	\begin{tblr}{
			vline{1,2,3,5,7,9} = {1-12}{},
			vline{4,6,8} = {3-12}{},
			cell{2}{3-8} = {c=1}{c},
			hline{1,3,5,7,9,11,13} = {-}{},	
		}
		\SetCell[c=1,r=2]{c}\backslashbox{$\Delta$}{$\ve$}	&  & \SetCell[c=2]{c}1/10 & & \SetCell[c=2]{c}1/20 & & \SetCell[c=2]{c} 1/100 &\\
		& & $\widehat{\theta}_\ve$ & $\widetilde{\theta}_\ve$ 
		& $\widehat{\theta}_\ve$ & $\widetilde{\theta}_\ve$ & $\widehat{\theta}_\ve$ & $\widetilde{\theta}_\ve$ \\
		1/5   & mean        &  (-0.91, 0.02) & (-0.90, 0.90) &  (-0.89, 0.90)& (-0.89, 0.89)   &  (-0.89, 0.86) & (-0.89, 0.89)     \\
		(k=20) &  resc. std.$^*$   &  (1.4, 1.2) &  (1.4, 1.2)  &  (1.5, 1.2)& (1.5, 1.2)    &  (1.5, 1.3)  & (1.5, 1.3)   \\
		1/10   & mean        &  (-0.98, 0.98) & (-0.96, 0.96) &  (-0.96, 0.96)& (-0.95, 0.95)   &  (-0.95, 0.95) & (-0.94, 0.95)   \\
		(k=10) &  resc. std.    &  (1.6, 1.4) &   (1.6, 1.4) &  (1.7, 1.5)& (1.6, 1.4)   &  (1.6, 1.4) & (1.6, 1.4)  \\
		1/20  &  mean       & (-1.05, 1.04) & (-1.01, 1.00)  & (-0.99, 0.99)&  (-0.98, 0.98)   & (-0.98, 0.98) &  (-0.91, 0.94)   \\
		(k=5)	&  resc. std.  & (1.7, 1.4) & (1.6, 1.4)  &  (1.7, 1.5) & (1.7, 1.5)    &  (1.8, 1.5)  & (1.7, 1.5) \\
		1/50  &mean      & (-1.11, 1.09) & (-1.03, 1.03)  &  (-1.03, 1.02)& (-1.00, 1.00)   &  (-0.99, 0.99)   & (-0.99, 0.99) \\
		(k=2)
		&  resc. std.    &   (1.8, 1.6) & (1.8, 1.5)  &  (1.8, 1.6)&  (1.8, 1.5)  &   (1.8, 1.6) & (1.8, 1.5)\\
		1/100 & mean       &  (-1.18, 1.15) &(-1.05, 1.04)  &  (-1.05, 1.04)&  (-1.01, 1.01)   & (-1.00, 1.00)&  (-1.00, 1.00)    \\
		(k=1)  & resc. std.     &  (2.0, 1.7) & (1.8, 1.6)  &  (1.9, 1.6)&(1.9, 1.5)    &    (1.8, 1.6) & (1.8, 1.6)
	\end{tblr}\\
	\footnotesize{$^*$ resc. std. is the value of the empirical standard deviation multiplied by $1/\ve$.}
\end{table}

%
%
%

\begin{table}[p]\caption{$\alpha=0.8$, $(\theta_0,\theta_1)=(-1,1)$, $T=50$, $h=10^{-2}$} \label{T : T 50 n 5000New}
	\centering \footnotesize
	\begin{tblr}{
			vline{1,2,3,5,7,9} = {1-10}{},
			vline{4,6,8} = {3-10}{},
			cell{2}{3-8} = {c=1}{c},
			hline{1,3,5,7,9,11} = {-}{},	
		}
		\SetCell[c=1,r=2]{c}\backslashbox{$\Delta$}{$\ve$}	&  & \SetCell[c=2]{c}1/10 & & \SetCell[c=2]{c}1/20 & & \SetCell[c=2]{c} 1/100 &\\
		& & $\widehat{\theta}_\ve$ & $\widetilde{\theta}_\ve$ 
		& $\widehat{\theta}_\ve$ & $\widetilde{\theta}_\ve$ & $\widehat{\theta}_\ve$ & $\widetilde{\theta}_\ve$ \\
		1/10   & mean        &  (-0.97, 0.97) & (-0.93, 0.93) &  (-0.96, 0.96)& (-0.95, 0.95)   &  (-0.95, 0.95) & (-0.95, 0.95)   \\
		(k=10) &  resc. std.    &  (1.1, 1.0) &   (1.1, 1.0) &  (1.2, 1.2)& (1.2, 1.2)   &  (1.3, 1.2) & (1.3, 1.2)  \\
		1/20  &  mean       & (-1.05, 1.05) & (-0.96, 0.97)  & (-1.00, 1.00)&  (-0.97, 0.97)   & (-0.98, 0.98) &  (-0.98, 0.98)   \\
		(k=5)	&  resc. std.  & (1.2, 1.1) & (1.1, 1.0)  &  (1.3, 1.3) & (1.3, 1.2)    &  (1.4, 1.3)  & (1.4, 1.3) \\
		1/50  &mean      & (-1.18, 1.17) & (-1.00, 1.00)  &  (-1.06, 1.06)& (-1.00, 1.00)   &  (-1.00, 1.00)   & (-0.99, 0.99) \\
		(k=2)
		&  resc. std.    &   (1.3, 1.3) & (1.3, 1.2)  &  (1.5, 1.5)&  (1.4, 1.4)  &   (1.4, 1.3) & (1.4, 1.3)\\
		1/100 & mean       &  (-1.31, 1.29) &(-1.00, 1.00)  &  (-1.11, 1.10)&  (-1.00, 1.00)   & (-1.00, 1.00)&  (-1.00, 1.00)    \\
		(k=1)  & resc. std.     &  (1.5, 1.5) & (1.2, 1.1)  &  (1.5, 1.5)&(1.4, 1.3)    &    (1.5, 1.4) & (1.5, 1.4)
	\end{tblr}\\
	\footnotesize{$^*$ resc. std. is the value of the empirical standard deviation multiplied by $1/\ve$.}
\end{table}

%
%
%

\begin{table}[p]\caption{$\alpha=0.95$, $(\theta_0,\theta_1)=(-1,1)$, $T=1$, $h=10^{-2}$}  \label{T : T 1 n 100 alpha095New}
	\centering \footnotesize
	\begin{tblr}{
			vline{1,2,3,5,7,9} = {1-12}{},
			vline{4,6,8} = {3-12}{},
			cell{2}{3-8} = {c=1}{c},
			hline{1,3,5,7,9,11,13} = {-}{},	
		}
		\SetCell[c=1,r=2]{c}\backslashbox{$\Delta$}{$\ve$}	&  & \SetCell[c=2]{c}1/10 & & \SetCell[c=2]{c}1/20 & & \SetCell[c=2]{c} 1/100 &\\
		& & $\widehat{\theta}_\ve$ & $\widetilde{\theta}_\ve$ 
		& $\widehat{\theta}_\ve$ & $\widetilde{\theta}_\ve$ & $\widehat{\theta}_\ve$ & $\widetilde{\theta}_\ve$ \\
		1/5   & mean        &  (-0.95, 0.91) & (-0.93, 0.91) &  (-0.90, 0.91)& (-0.90, 0.90)   &  (-0.89, 0.86) & (-0.88, 0.91)     \\
		(k=20) &  resc. std.$^*$   &  (5.9, 1.7) &  (5.9, 1.7)  &  (5.8, 1.8)& (5.8, 1.8)    &  (5.9, 1.8)  & (5.9, 1.8)   \\
		1/10   & mean        &  (-1.07, 0.99) & (-1.05, 0.99) &  (-0.99, 0.96)& (-0.98, 0.96)   &  (-0.94, 0.95) & (-0.93, 0.95)   \\
		(k=10) &  resc. std.    &  (5.9, 2.1) &   (5.9, 2.1) &  (5.5, 2.1)& (5.5, 2.1)   &  (5.3, 2.0) & (5.3, 2.0)  \\
		1/20  &  mean       & (-1.17, 1.04) & (-1.15, 1.04)  & (-1.02, 0.99)&  (-1.01, 0.99)   & (-0.98, 0.98) &  (-0.97, 0.97)   \\
		(k=5)	&  resc. std.  & (5.9, 2.2) & (5.9, 2.2)  &  (5.6, 2.2) & (5.5, 2.2)    &  (5.8, 2.3)  & (5.5, 2.2) \\
		1/50  &mean      & (-1.18, 1.06) & (-1.15, 1.06)  &  (-1.02, 1.00)& (-1.00, 0.99)   &  (-1.00, 0.99)   & (-0.99, 0.99) \\
		(k=2)
		&  resc. std.    &   (6.0, 2.4) & (6.0, 2.3)  &  (5.6, 2.3)&  (5.6, 2.3)  &   (5.6, 2.3) & (5.5, 2.3)\\
		1/100 & mean       &  (-1.20, 1.08) &(-1.17, 1.07)  &  (-1.05, 1.02)&  (-1.04, 1.01)   & (-1.00, 1.00)&  (-0.99, 1.00)    \\
		(k=1)  & resc. std.     &  (6.0, 2.4) & (6.0, 2.4)  &  (5.7, 2.4)&(5.7, 2.4)    &    (5.6, 2.3) & (6.3, 2.7)
	\end{tblr}\\
	\footnotesize{$^*$ resc. std. is the value of the empirical standard deviation multiplied by $1/\ve$.}
\end{table}

%
%
%

\begin{table}[p]\caption{$\alpha=0.95$, $(\theta_0,\theta_1)=(-1,1)$, $T=10$, $h=10^{-2}$}  \label{T : T 10 n 1000 alpha095New}
	\centering \footnotesize
	\begin{tblr}{
			vline{1,2,3,5,7,9} = {1-12}{},
			vline{4,6,8} = {3-12}{},
			cell{2}{3-8} = {c=1}{c},
			hline{1,3,5,7,9,11,13} = {-}{},	
		}
		\SetCell[c=1,r=2]{c}\backslashbox{$\Delta$}{$\ve$}	&  & \SetCell[c=2]{c}1/10 & & \SetCell[c=2]{c}1/20 & & \SetCell[c=2]{c} 1/100 &\\
		& & $\widehat{\theta}_\ve$ & $\widetilde{\theta}_\ve$ 
		& $\widehat{\theta}_\ve$ & $\widetilde{\theta}_\ve$ & $\widehat{\theta}_\ve$ & $\widetilde{\theta}_\ve$ \\
		1/5   & mean        &  (-0.94, 0.94) & (-0.94, 0.93) &  (-0.91, 0.91)& (-0.91, 0.91)   &  (-0.90, 0.91) & (-0.90, 0.90)     \\
		(k=20) &  resc. std.$^*$   &  (1.4, 1.2) &  (1.4, 1.2)  &  (1.3, 1.2)& (1.3, 1.2)    &  (1.3, 1.2)  & (1.3, 1.2)   \\
		1/10   & mean        &  (-0.99, 0.99) & (-0.99, 0.98) &  (-0.96, 0.96)& (-0.96, 0.96)   &  (-0.95, 0.95) & (-0.93, 0.95)   \\
		(k=10) &  resc. std.    &  (1.5, 1.3) &   (1.5, 1.3) &  (1.5, 1.3)& (1.5, 1.3)   &  (1.5, 1.4) & (1.5, 1.4)  \\
		1/20  &  mean       & (-1.02, 1.01) & (-1.01, 1.00)  & (-0.99, 0.99)&  (-0.99, 0.99)   & (-0.98, 0.98) &  (-0.97, 0.97)   \\
		(k=5)	&  resc. std.  & (1.5, 1.4) & (1.5, 1.4)  &  (1.6, 1.4) & (1.5, 1.5)    &  (1.5, 1.4)  & (1.5, 1.4) \\
		1/50  &mean      & (-1.04, 1.04) & (-1.03, 1.03)  &  (-1.00, 1.00)& (-1.00, 1.00)   &  (-0.99, 0.99)   & (-0.99, 0.99) \\
		(k=2)
		&  resc. std.    &   (1.6, 1.5) & (1.6, 1.5)  &  (1.5, 1.5)&  (1.5, 1.4)  &   (1.7, 1.5 ) & (1.7, 1.5)\\
		1/100 & mean       &  (-1.05, 1.05) &(-1.04, 1.03)  &  (-1.00, 1.00)&  (-1.00, 1.00)   & (-1.00, 1.00)&  (-1.00, 1.00)    \\
		(k=1)  & resc. std.     &  (1.7, 1.5) & (1.6, 1.5)  &  (1.6, 1.5)&(1.6, 1.5)    &    (1.7, 1.5) & (1.7, 1.5)
	\end{tblr}\\
	\footnotesize{$^*$ resc. std. is the value of the empirical standard deviation multiplied by $1/\ve$.}
\end{table}

%
%
%

\begin{table}[p]\caption{$\alpha=0.95$, $(\theta_0,\theta_1)=(-1,1)$, $T=50$, $h=10^{-2}$} \label{T : T 50 n 5000 alpha095New}
	\centering \footnotesize
	\begin{tblr}{
			vline{1,2,3,5,7,9} = {1-10}{},
			vline{4,6,8} = {3-10}{},
			cell{2}{3-8} = {c=1}{c},
			hline{1,3,5,7,9,11} = {-}{},	
		}
		\SetCell[c=1,r=2]{c}\backslashbox{$\Delta$}{$\ve$}	&  & \SetCell[c=2]{c}1/10 & & \SetCell[c=2]{c}1/20 & & \SetCell[c=2]{c} 1/100 &\\
		& & $\widehat{\theta}_\ve$ & $\widetilde{\theta}_\ve$ 
		& $\widehat{\theta}_\ve$ & $\widetilde{\theta}_\ve$ & $\widehat{\theta}_\ve$ & $\widetilde{\theta}_\ve$ \\
		1/10   & mean        &  (-0.97, 0.97) & (-0.97, 0.97) &  (-0.96, 0.96)& (-0.96, 0.96)   &  (-0.95, 0.95) & (-0.95, 0.95)   \\
		(k=10) &  resc. std.    &  (1.1, 1.1) &   (1.1, 1.1) &  (1.2, 1.2)& (1.2, 1.2)   &  (1.3, 1.3) & (1.3, 1.3)  \\
		1/20  &  mean       & (-1.01, 1.01) & (-0.99, 0.99)  & (-0.99, 0.99)&  (-0.98, 0.98)   & (-0.98, 0.98) &  (-0.98, 0.98)   \\
		(k=5)	&  resc. std.  & (1.2, 1.2) & (1.2, 1.2)  &  (1.3, 1.3) & (1.3, 1.3)    &  (1.4, 1.4)  & (1.4, 1.4) \\
		1/50  &mean      & (-1.04, 1.04) & (-1.02, 1.02)  &  (-1.01, 1.01)& (-1.00, 1.00)   &  (-1.00, 1.00)   & (-0.99, 0.99) \\
		(k=2)
		&  resc. std.    &   (1.2, 1.2) & (1.2, 1.1)  &  (1.4, 1.3)&  (1.4, 1.3)  &   (1.4, 1.4) & (1.4, 1.4)\\
		1/100 & mean       &  (-1.06, 1.06) &(-1.02, 1.02)  &  (-1.02, 1.02)&  (-1.01, 1.01)   & (-1.00, 1.00)&  (-0.99, 0.99)    \\
		(k=1)  & resc. std.     &  (1.2, 1.2) & (1.2, 1.2)  &  (1.4, 1.4)&(1.4, 1.4)    &    (1.4, 1.4) & (1.5, 1.4)
	\end{tblr}\\
	\footnotesize{$^*$ resc. std. is the value of the empirical standard deviation multiplied by $1/\ve$.}
\end{table}

%
%

\begin{table}[p]\caption{$\alpha=0.6$, $(\theta_0,\theta_1)=(-1,1)$, $T=1$, $h=10^{-2}$}  \label{T : T 1 n 100 alpha06New}
	\centering \footnotesize
	\begin{tblr}{
			vline{1,2,3,5,7,9} = {1-12}{},
			vline{4,6,8} = {3-12}{},
			cell{2}{3-8} = {c=1}{c},
			hline{1,3,5,7,9,11,13} = {-}{},	
		}
		\SetCell[c=1,r=2]{c}\backslashbox{$\Delta$}{$\ve$}	&  & \SetCell[c=2]{c}1/10 & & \SetCell[c=2]{c}1/20 & & \SetCell[c=2]{c} 1/100 &\\
		& & $\widehat{\theta}_\ve$ & $\widetilde{\theta}_\ve$ 
		& $\widehat{\theta}_\ve$ & $\widetilde{\theta}_\ve$ & $\widehat{\theta}_\ve$ & $\widetilde{\theta}_\ve$ \\
		1/5   & mean        &  (-0.67, 0.82) & (-0.64, 0.80) &  (-0.80, 0.88)& (-0.78, 0.86)   &  (-0.62, 0.80) & (-0.61, 0.79)     \\
		(k=20) &  resc. std.$^*$   &  (4.5, 1.7) &  (4.4, 1.6)  &  (5.7, 2.3)& (5.4, 2.1)    &  (4.7, 1.7)  & (4.6, 1.7)   \\
		1/10   & mean        &  (-0.95, 0.94) & (-0.86, 0.91) &  (-0.81, 0.89)& (-0.78, 0.88)   &  (-0.79, 0.88) & (-0.77, 0.88)   \\
		(k=10) &  resc. std.    &  (5.3, 2.2) &   (5.0, 2.1) &  (5.9, 2.2)& (5.8, 2.2)   &  (5.7, 2.4) & (5.4, 2.3)  \\
		1/20  &  mean       & (-1.22, 1.07) & (-1.01, 0.99)  & (-1.00, 0.98)&  (-0.91, 0.94)   & (-0.90, 0.95) &  (-0.87, 0.93)   \\
		(k=5)	&  resc. std.  & (6.0, 2.5) & (5.4, 2.4)  &  (6.5, 2.8) & (6.2, 2.7)    &  (6.4, 2.7)  & (6.1, 2.6) \\
		1/50  &mean      & (-1.65, 1.27) & (-1.24, 1.06)  &  (-1.19, 1.07)& (-1.02, 1.00)   &  (-0.99, 0.99)   & (-0.95, 0.97) \\
		(k=2)
		&  resc. std.    &   (7.5, 3.4) & (6.1, 2.8)  &  (7.6, 3.3)&  (6.9, 3.0)  &   (7.5, 3.4) & (7.2, 3.2)\\
		1/100 & mean       &  (-2.1, 1.45) &(-1.17, 1.08)  &  (-1.36, 1.15)&  (-1.05, 1.02)   & (-1.03, 1.01)&  (-0.98, 0.99)    \\
		(k=1)  & resc. std.     &  (8.9, 3.8) & (6.4, 2.9)  &  (8.2, 3.6)&(7.0, 3.2)    &    (7.9, 3.5) & (7.5, 3.3)
	\end{tblr}\\
	\footnotesize{$^*$ resc. std. is the value of the empirical standard deviation multiplied by $1/\ve$.}
\end{table}

%
%

%
%

\begin{table}[p]\caption{$\alpha=0.6$, $(\theta_0,\theta_1)=(-1,1)$, $T=10$, $h=10^{-2}$}  \label{T : T 10 n 1000 alpha06New}
	\centering \footnotesize
	\begin{tblr}{
			vline{1,2,3,5,7,9} = {1-12}{},
			vline{4,6,8} = {3-12}{},
			cell{2}{3-8} = {c=1}{c},
			hline{1,3,5,7,9,11,13} = {-}{},	
		}
		\SetCell[c=1,r=2]{c}\backslashbox{$\Delta$}{$\ve$}	&  & \SetCell[c=2]{c}1/10 & & \SetCell[c=2]{c}1/20 & & \SetCell[c=2]{c} 1/100 &\\
		& & $\widehat{\theta}_\ve$ & $\widetilde{\theta}_\ve$ 
		& $\widehat{\theta}_\ve$ & $\widetilde{\theta}_\ve$ & $\widehat{\theta}_\ve$ & $\widetilde{\theta}_\ve$ \\
		1/5   & mean        &  (-0.79, 0.83) & (-0.75, 0.80) &  (-0.83, 0.86)& (-0.81, 0.85)   &  (-0.84, 0.87) & (-0.83, 0.87)     \\
		(k=20) &  resc. std.$^*$   &  (1.4, 1.0) &  (1.3, 1.0)  &  (1.5, 1.2)& (1.5, 1.2)    &  (1.7, 1.3)  & (1.7, 1.3)   \\
		1/10   & mean        &  (-0.91, 0.93) & (-0.82, 0.86) &  (-0.91, 0.93)& (-0.89, 0.91)   &  (-0.92, 0.93) & (-0.91, 0.93)   \\
		(k=10) &  resc. std.    &  (1.6, 1.3) &   (1.4, 1.1) &  (1.7, 1.4)& (1.7, 1.3)   &  (1.9, 1.5) & (1.9, 1.5)  \\
		1/20  &  mean       & (-1.07, 1.05) & (-0.89, 0.91)  & (-0.99, 0.99)&  (-0.93, 0.94)   & (-0.96, 0.97) &  (-0.96, 0.97)   \\
		(k=5)	&  resc. std.  & (2.0, 1.5) & (1.7, 1.3)  &  (2.0, 1.5) & (1.9, 1.6)    &  (2.2, 1.7)  & (2.2, 1.7) \\
		1/50  &mean      & (-1.35, 1.27) & (-0.92, 0.94)  &  (-1.11, 1.08)& (-0.97, 0.97)   &  (-0.99, 0.99)   & (-0.98, 0.99) \\
		(k=2)
		&  resc. std.    &   (2.3, 1.8) & (1.7, 1.3)  &  (2.4, 1.9 )&  (2.1, 1.7)  &   (2.3, 1.8) & (2.3, 1.8)\\
		1/100 & mean       &  (-1.71, 1.54) &(-0.97, 0.97)  &  (-1.23, 1.18)&  (-0.99, 0.99)   & (-1.00, 1.00)&  (-0.99, 0.99)    \\
		(k=1)  & resc. std.     &  (3.0, 2.3) & (1.8, 1.4)  &  (2.7, 2.1)&(2.1, 1.7)    &    (2.3, 1.7) & (2.2, 1.7)
	\end{tblr}\\
	\footnotesize{$^*$ resc. std. is the value of the empirical standard deviation multiplied by $1/\ve$.}
\end{table}

%
%
%
%
%

\begin{table}[p]\caption{$\alpha=0.6$, $(\theta_0,\theta_1)=(-1,1)$, $T=50$, $h=10^{-2}$} \label{T : T 50 n 5000 alpha060New}
	\centering \footnotesize
	\begin{tblr}{
			vline{1,2,3,5,7,9} = {1-10}{},
			vline{4,6,8} = {3-10}{},
			cell{2}{3-8} = {c=1}{c},
			hline{1,3,5,7,9,11} = {-}{},	
		}
		\SetCell[c=1,r=2]{c}\backslashbox{$\Delta$}{$\ve$}	&  & \SetCell[c=2]{c}1/10 & & \SetCell[c=2]{c}1/20 & & \SetCell[c=2]{c} 1/100 &\\
		& & $\widehat{\theta}_\ve$ & $\widetilde{\theta}_\ve$ 
		& $\widehat{\theta}_\ve$ & $\widetilde{\theta}_\ve$ & $\widehat{\theta}_\ve$ & $\widetilde{\theta}_\ve$ \\
		1/10   & mean        &  (-0.86, 0.87) & (-0.72, 0.74) &  (-0.91, 0.92)& (-0.86, 0.87)   &  (-0.94, 0.95) & (-0.95, 0.95)   \\
		(k=10) &  resc. std.    &  (0.9, 0.8) &   (0.8, 0.7) &  (1.2, 1.1)& (1.1, 1.0)   &  (1.3, 1.2) & (1.3, 1.2)  \\
		1/20  &  mean       & (-1.06, 1.06) & (-0.78, 0.80)  & (-1.01, 1.01)&  (-0.90, 0.91)   & (-0.98, 0.98) &  (-0.97, 0.97)   \\
		(k=5)	&  resc. std.  & (1.1, 1.0) & (0.8, 0.9)  &  (1.4, 1.2) & (1.2, 1.1)    &  (1.5, 1.4)  & (1.5, 1.3) \\
		1/50  &mean      & (-1.48, 1.43) & (-0.85, 0.86)  &  (-1.20, 1.18)& (-0.93, 0.94)   &  (-1.00, 1.00)   & (-0.99, 0.99) \\
		(k=2)
		&  resc. std.    &   (1.5, 1.4) & (0.9, 0.8)  &  (1.6, 1.5)&  (1.4, 1.3)  &   (1.5, 1.3) & (1.5, 1.3)\\
		1/100 & mean       &  (-2.00, 1.90) &(-0.89, 0.90)  &  (-1.42, 1.38)&  (-0.95, 0.96)   & (-1.02, 1.02)&  (-0.99, 0.99)    \\
		(k=1)  & resc. std.     &  (1.9, 1.7) & (0.9, 0.8)  &  (1.9, 1.8)&(1.2, 1.2)    &    (1.5, 1.4) & (1.5, 1.4)
	\end{tblr}\\
	\footnotesize{$^*$ resc. std. is the value of the empirical standard deviation multiplied by $1/\ve$.}
\end{table}

\printbibliography

\end{document}